\documentclass[a4paper,12pt]{article}

\usepackage{authblk}
\usepackage{amssymb, amsmath, amsthm}
\usepackage{graphicx, tikz, subcaption}
\usetikzlibrary{arrows.meta, positioning, backgrounds}
\usepackage{a4wide}

\newtheorem{theorem}{Theorem}[section]
\newtheorem{conjecture}[theorem]{Conjecture}
\newtheorem{corollary}[theorem]{Corollary}

\newtheorem{proposition}[theorem]{Proposition}
\theoremstyle{definition}
\newtheorem{definition}[theorem]{Definition}
\newtheorem{remark}[theorem]{Remark}
\newtheorem{observation}[theorem]{Observation}

\begin{document}

\title{{\bf Solid-Resolving Sets on Directed Graphs}\footnote{This paper has been presented in part in IWOCA 2026~\cite{IWOCA2026}.}}
\author[1,2]{Anni Hakanen\footnote{Corresponding author. Email: \texttt{anehak@utu.fi}.}}
\author[1,3]{Pavan P D}

\affil[1]{Department of Mathematics and Statistics, University of Turku, Finland}
\affil[2]{Turku Collegium for Science, Medicine and Technology (TCSMT), University of Turku, Finland}
\affil[3]{CIMFAV \& Instituto de Ingenier\'{i}a Matem\'{a}tica, Universidad de Valpara\'{i}so, Chile}

\date{}

\maketitle

\begin{abstract}
A set $S \subseteq V(G)$ is a solid-resolving set of an undirected graph $G$ if each vertex can be distinguished from any vertex subset by their distances from the set $S$. Solid-resolving sets are a generalization of resolving sets, which have been studied extensively in the literature. We extend this concept to oriented graphs and directed graphs. In this setting, we first define solid-resolving sets for digraphs and then consider vertices that must be included in any solid-resolving set. We also study the solid-metric dimension, which is the minimum cardinality of a solid-resolving set of a given digraph. We determine the possible values of the solid-metric dimension for several common graph classes, including tournaments, directed trees, oriented cycles, complete bipartite graphs, as well as for Cartesian products of oriented paths and cycles. We also show via some examples that there is no trivial relation between the solid-metric dimension of a digraph and the solid-metric dimension of its underlying undirected graph. Furthermore, we show that the solid-metric dimension can increase or decrease by an arbitrary $k$ when one vertex is removed. Similarly, removing a single arc can decrease the solid-metric dimension by an arbitrary~$k$; however, it can increase by at most 1. The arbitrary changes in solid-metric dimension indicate that the solid-metric dimension is a global, rather than a local, property.

\end{abstract}

\section{Introduction}
	
	A resolving set is a vertex set that distinguishes all vertices of the graph by looking at their distances from the vertices of the resolving set.
	Formally, we say that a vertex $v$ \emph{distinguishes} distinct vertices $x$ and $y$ of a graph $G$ if $d(v, x) \neq d(v, y)$, where $d(u, v)$ is the length of a shortest path from $u$ to $v$.
	A set $S \subseteq V(G)$ is a \emph{resolving set} of the graph~$G$ if, for every pair of distinct vertices $u,v \in V(G)$, there exists a vertex $s \in S$ such that $d(s, u) \neq d(s, v)$.
	The minimum cardinality of a resolving set of $G$ is called the \emph{metric dimension} of $G$.
	
	Resolving sets and metric dimension were introduced for undirected graphs in the 1970s by Slater~\cite{S:leavesTree} and Harary and Melter~\cite{Harary76}, and have since been studied extensively~\cite{KuziakYeroSurvey,TillquistSurvey}. Over the years, the study has expanded to variations of the original concept: for example, one can distinguish edges instead of vertices~\cite{KelencEdge18}, use a different distance than the traditional graph distance~\cite{JannesariAdjacencyLexi,SaenpholphatDetour}, or consider more general graphs by including weights~\cite{EpsteinWeighted} or edge directions~\cite{ChartrandDirectedMD}. See the surveys~\cite{KuziakYeroSurvey,TillquistSurvey} for more information on metric dimension and its variants.

	Considering metric dimension in the context of digraphs raises the question of how to define the distance between two vertices when there is no directed path from the first vertex to the second, i.e., the second vertex is not reached from the first vertex. Metric dimension was first defined for digraphs by Chartrand, Raines, and Zhang~\cite{ChartrandDirectedMD}, who dealt with this issue by avoiding it altogether! Their definition requires that every vertex of the resolving set is reachable from every vertex outside the resolving set. Consequently, resolving sets, and hence the metric dimension, are undefined for many digraphs under this definition. This definition of resolving sets in digraphs has its applications, however, in~\cite{FengDirectedMD,PancahayaniDirectedMD,RajanDirectedMD,SchmitzDirectedMD,VetrikDirectedMD}, for example.
	
	Another approach to answering the question was presented more recently in~\cite{AraujoOrientations,BensmailOrientedMD}: the distance from $u$ to $v$ is defined to be $\infty$ when $v$ is not reachable from $u$. In this case, the infinite distance can be used to distinguish vertices. If $x$ is reachable from $v$ but $y$ is not, then $d(v,x) < d(v,y)$ under this convention, and $v$ distinguishes $x$ and $y$. This interpretation of infinite distance makes it possible to extend the definitions of resolving sets and metric dimension to all digraphs. 
	
	The \emph{solid-metric dimension} was introduced for undirected graphs in~\cite{Solid}. Similar to classical resolving sets, solid-resolving sets can distinguish distinct vertices from each other; however, they can also distinguish any single vertex from a set of vertices. Let the distance from a vertex $s$ to a set of vertices $Y$ be defined as $d(s,Y) = \min_{y \in Y} d(s,y)$. For an undirected graph $G$, the set $S \subseteq V(G)$ is a solid-resolving set if, for every $x \in V(G)$ and every non-empty set $Y \subseteq V(G)$, with $Y \neq \{x\}$, there exists $s \in S$ such that $d(s,x) \neq d(s,Y)$.

    Before defining solid-resolving sets for digraphs, we first present some basic definitions and notation. A \emph{digraph} (or \emph{directed graph}) $D$ consists of a vertex set $V(D)$ and an arc set $A(D)$. An arc is an ordered pair of vertices, and an arc from $u$ to $v$ is denoted by $\overrightarrow{uv}$ or $(u,v)$. The digraphs considered in this paper are simple, that is, without loops $\overrightarrow{uu}$. An \emph{orientation} of an undirected graph $G$ is a digraph where each edge of $G$ is given a direction (and bidirected edges are forbidden). When it is necessary to distinguish an undirected graph $G$ from one of its orientations, we may use the arrow notation $\overrightarrow{G}$; however, we usually denote digraphs without the arrow.

    The \emph{distance} from $u$ to $v$ is defined as the number of arcs in a shortest directed path from $u$ to $v$, and is denoted by $d_D(u, v)$ (where the subscript $D$ may be omitted if the context is clear). If no such directed path exists, then $v$ is not \emph{reachable} from $u$, and $d(u, v) = \infty$. The distance from a vertex $u$ to a vertex set $X \subseteq V(D)$ is defined as $d(u, X) = \min_{x \in X} d(u, x)$. If no directed path exists from $u$ to any vertex in $X$, then $d(u, X) = \infty$.
    
	We now define solid-resolving sets for digraphs as follows.
	
	\begin{definition}\label{def-solid}
		Let $D$ be a digraph. A set $S \subseteq V(D)$ is a solid-resolving set of $D$ if for every vertex $v \in V(D)$ and every non-empty subset $X \subseteq V(D)$ with $X \neq \{v\}$, there exists $w \in S$ such that $d(w, v) \neq d(w, X)$. The minimum cardinality of a solid-resolving set is called the solid-metric dimension of $D$, and is denoted by $\beta_S(D)$.
	\end{definition}
	
    In some papers, resolving sets are defined by using distance arrays or metric representations. The same can be done with solid-resolving sets as follows. Let $S = \{s_1, \dots, s_k\} \subseteq V(D)$. The \emph{distance array} (or \emph{metric representation}) of a vertex $v$ with respect to $S$ is defined as $\mathcal{D}_S(v) = (d(s_1, v), \dots, d(s_k, v))$. Similarly, the distance array of a set $X$ with respect to $S$ is $\mathcal{D}_S(X) = (d(s_1, X), \dots, d(s_k, X))$. Now, the set $S \subseteq V(D)$ is a solid-resolving set of $D$ if, for every vertex $v \in V(D)$ and every non-empty subset $X \subseteq V(D)$, the equality $\mathcal{D}_S(v) = \mathcal{D}_S(X)$ implies that $X = \{v\}$.
	
	In this work, we study solid-resolving sets and the solid-metric dimension of digraphs. In Section~\ref{sec-basicresults}, we generalize some basic results for solid-resolving sets from undirected graphs to digraphs. Our exposition places special emphasis on \emph{forced vertices}, that is, vertices which belong to every solid-resolving set. In Section~\ref{sec-graphfamilies}, we consider digraphs and orientations whose underlying graphs are trees, cycles, complete bipartite graphs or complete graphs. We establish bounds on their solid-metric dimensions and determine the exact values in some cases. In Section~\ref{sec-graphoperations}, we prove that removing a vertex can change the solid-metric dimension by an arbitrary $k$. We prove that removing an arc can decrease the solid-metric dimension by an arbitrarily large amount, whereas any increase is at most~1. In Section~\ref{sec-smdundirected}, we show that the difference between the solid-metric dimension of a digraph and that of its underlying undirected graph can be arbitrarily large. Finally, in Section~\ref{sec-cartesianproducts}, we consider the solid-metric dimension of Cartesian products of oriented paths and cycles. Section~\ref{sec-conclusions} concludes the paper.

\section{Basic Results}\label{sec-basicresults}

We begin by generalizing some basic results on solid-resolving sets from~\cite{Solid} to digraphs. The following theorem gives us a useful characterization of solid-resolving sets, which we shall make use of extensively henceforth.

\begin{theorem}
\label{thm:srsmain}
    Let $D$ be a nontrivial digraph. A set $S \subseteq V(D)$ is a solid-resolving set of $D$ if and only if for all distinct $u, v \in V(D)$, there is an element $s \in S$ such that $d(s, u) < d(s, v)$.
\end{theorem}
\begin{proof}
    Suppose first that $S$ is a solid-resolving set. To prove the claim, let us assume on the contrary that there exist vertices $u, v \in V(D)$ such that $d(s, v) \leq d(s, u)$ for all $s \in S$. Let $X = \{u, v\}$. This implies that $d(s, X) = \min (d(s, u), d(s, v)) = d(s, v)$ for all $s \in S$. This contradicts the assumption that $S$ is a solid-resolving set.

    To prove the converse, let $v \in V(D)$ and let $X \subseteq V(D)$ be a non-empty set, with $X \neq \{v\}$. Let $u \in X \setminus \{v\}$. By the hypothesis, there exists a vertex $s \in S$ such that $d(s, u) < d(s, v)$. But since $d(s, X) \leq d(s, u)$, we have $d(s, X) < d(s, v)$, and thus, $d(s, X) \neq d(s, v)$. 
    Hence, $S$ is a solid-resolving set.  
\end{proof}

Notice that Theorem~\ref{thm:srsmain} implies that each vertex of the digraph is reachable from some vertex of its solid-resolving set. Thus, if the underlying undirected graph of $D$ is not connected, we can determine the solid-resolving sets of $D$ by looking at these components separately. Indeed, the union of solid-resolving sets of the components is always a solid-resolving set of $D$. Moreover, if $S$ is a solid-resolving set of $D$ and $D_i$ is a component of~$D$, then $S\cap V(D_i)$ must be a solid-resolving set of $D_i$ since the condition $d(s,u)<d(s,v)$ cannot be satisfied for vertex pairs $u,v \in V(D_i)$ by any $s \notin V(D_i)$. In particular, if $D_i \simeq K_1$, then this singleton vertex must always be included in any solid-resolving set of $D$. Therefore, we mainly focus on nontrivial digraphs whose underlying graphs are connected.

\begin{remark}[Computational complexity]
    When using Definition~\ref{def-solid}, verifying that a given set is a solid-resolving set requires one to compare each vertex to $2^n-2$ sets. This leads to an algorithm with exponential time complexity. However, we need only $\mathcal{O}(n^2)$ comparisons when using Theorem~\ref{thm:srsmain}. Thus, there exists a polynomial-time algorithm that checks whether a given set is a solid-resolving set of the given digraph. Consequently, the decision problem asking whether $\beta_S(D) \leq k$ for a given digraph $D$ and integer $k$ is in NP. The same decision problem was proven to be NP-hard for undirected graphs in~\cite{Solid}. Since undirected graphs are a subclass of digraphs, it follows that this decision problem is NP-complete for digraphs.
\end{remark}

In~\cite{Solid}, it was noted that some vertices must belong to all solid-resolving sets, and thus they can be used to prove lower bounds on the solid-metric dimension. We make a similar observation in the context of digraphs.

\begin{definition}
    A vertex $u \in V(D)$ is called a \emph{forced vertex} of $D$ with respect to solid-resolving sets if every solid-resolving set of $D$ contains $u$.
\end{definition}

The \emph{open in-neighborhood} $N^-_{D}(v)$ (resp. \emph{open out-neighborhood} $N^+_{D}(v)$) of a vertex~$v$ is defined as $N^-_{D}(v) = \{u \in V(D) \mid \overrightarrow{uv} \in A(D)\}$ (resp., $N^+_{D}(v) = \{u \in V(D) \mid \overrightarrow{vu} \in A(D)\}$), and the \emph{closed in-neighborhood} $N^-_{D}[v]$ (\emph{closed out-neighborhood} $N^+_{D}[v]$) as $N^-_{D}[v] = N^-_{D}(v) \cup \{v\}$ (resp., $N^+_{D}[v] = N^+_{D}(v) \cup \{v\}$). The \emph{in-degree} $\deg^-_{D}(v)$ (resp., \emph{out-degree} $\deg^+_{D}(v)$) of a vertex $v$ is the cardinality of the open in-neighborhood $N^-_{D}(v)$ (resp., open out-neighborhood $N^+_{D}(v)$). Whenever the context is clear, we usually drop the subscript~$D$ for convenience. A vertex is called a \emph{source} if it has no in-neighbors, and a \emph{sink} if it has no out-neighbors.

The next theorem gives us a characterization of forced vertices. In the proof, we implicitly use the fact that any superset of a solid-resolving set is also a solid-resolving set.

\begin{theorem}\label{thm:forced-characterization}
    Let $D$ be a non-trivial digraph. A vertex $u \in V(D)$ is a forced vertex of $D$ if and only if there exists a vertex $v \in V(D) \setminus \{u\}$ such that $N^-(u) \subseteq N^-[v]$.
\end{theorem}
\begin{proof}
    First, suppose that $u$ and $v$ are distinct vertices in $D$ with $N^-(u) \subseteq N^-[v]$. We will show that the set $S = V(D) \setminus \{u\}$ is not a solid-resolving set, which implies that every solid-resolving set of $D$ must include $u$. Assume, to the contrary, that $S$ is a solid-resolving set. By Theorem~\ref{thm:srsmain}, there exists an element $s \in S$ such that $d(s, u) < d(s, v)$. Since $s \neq u$, any shortest path from $s$ to $u$ must pass through some $w \in N^-(u)$, so we can write $d(s, u) = d(s, w) + 1$. But since $w \in N^-(u) \subseteq N^-[v]$, we have $d(s, v) \leq d(s, w) + 1 = d(s, u)$, which is a contradiction. Thus, the set $S$ and any of its subsets are not solid-resolving sets.

    Conversely, assume that $u$ is a forced vertex of $D$, but $N^-(u) \nsubseteq N^-[v]$ for all $v \in V(D) \setminus \{u\}$. We will show that $S = V(D) \setminus \{u\}$ is a solid-resolving set, contradicting our assumption that $u$ is forced. For any distinct $s \in S$ and $t \in V(D) \setminus \{s\}$ we trivially have $d(s, s) < d(s, t)$. Thus, it remains to distinguish $u$ from any $s \in S$. Since $N^-(u) \nsubseteq N^-[s]$ for any $s \in S$, we can choose a vertex, say $t \in N^-(u) \setminus N^-[s]$. Then, $d(t, u) = 1$ and $d(t, s) \geq 2$, and thus $d(t, u) < d(t, s)$. Hence, by Theorem~\ref{thm:srsmain}, $S$ is a solid-resolving set, a contradiction.
\end{proof}

\begin{corollary}\label{cor:forceddegleq1}
    Let $D$ be a non-trivial digraph. A vertex $u \in V(D)$ with $\deg^-(u) \leq 1$ is a forced vertex of $D$.
\end{corollary}
\begin{proof}
    If $\deg^-(u) = 0$, then for any $v \in V(D) \setminus \{u\}$, we have $N^-(u) \subseteq N^-[v]$, and hence by Theorem~\ref{thm:forced-characterization}, $u$ is a forced vertex. Let $\deg^-(u) = 1$ and $N^-(u) = \{v\}$. Then $N^-(u) (= \{v\}) \subseteq (\{v\} \cup N^-(v)) = N^-[v]$. Hence, by Theorem~\ref{thm:forced-characterization}, $u$ is a forced vertex of~$D$.
\end{proof}

In~\cite{Solid}, the relationship between solid-resolving sets and the boundary of a graph was also explored. A definition of the boundary for digraphs was presented in~\cite{ChangatBoundary}. However, we present a different definition for the boundary, which is more along the lines of the original definition for undirected graphs in~\cite{ChartrandBoundary}. This definition allows us to generalize results regarding solid-resolving sets and the boundary from the undirected case.

The \emph{boundary} of a digraph $D$ is the set $\partial(D) = \{v \in V(D) \mid \exists u \in V(D), \forall w \in N^+(v) : d(u, w) \leq d(u, v)\}$. Vertices in $\partial(D)$ are called \emph{boundary vertices}. The boundary of a graph is always non-empty, for if the maximum distance between any pair of vertices in $D$ is equal to $d(u, v)$, then $v$ will always be a boundary vertex. That is, $|\partial(D)| \geq 1$. Notice that if $v$ is not reachable from some $u$, that is, if $d(u,v) = \infty$, then $v \in \partial(D)$. Thus, if the underlying undirected graph of $D$ is not connected, then $\partial(D) = V(D)$.

The \emph{transpose} of a digraph $D$ is another digraph $D^T$ obtained by reversing all arcs, that is, $\overrightarrow{uv} \in A(D)$ if and only if $\overrightarrow{vu} \in A(D^T)$.

\begin{theorem}
Let $D$ be a finite digraph.
    \begin{enumerate}
        \item The boundary of $D^T$ is a solid-resolving set of $D$.
        \item Every forced vertex of $D$ is a boundary vertex of $D^T$.
        \item There exist graphs that have solid-resolving sets of minimum cardinality that are not subsets of the boundary of the transpose. 
    \end{enumerate}
\end{theorem}
\begin{proof}
1. Let $u, v \in V(D)$ be distinct vertices. By Theorem~\ref{thm:srsmain}, to prove that $\partial(D^T)$ is a solid-resolving set, it suffices to find a vertex $s \in \partial(D^T)$ such that $d_D(s, u) < d_D(s, v)$, or equivalently $d_{D^T}(u, s) < d_{D^T}(v, s)$.
    
If $u \in \partial(D^T)$, since it is trivially true that $d_D(u, u) < d_D(u, v)$, we are done. So let $u \notin \partial(D^T)$. By definition of the boundary, for every vertex $t$, there exists a vertex $w \in N^+(u)$ such that $d_{D^T}(t, w) > d_{D^T}(t, u)$. In particular, we have $d_{D^T}(v, w) > d_{D^T}(v, u)$, which implies $d_{D^T}(v, w) \geq d_{D^T}(v, u) + 1 = d_{D^T}(v, u) + d_{D^T}(u, w)$, since $w \in N^+(u)$. But by the definition of distance, we have $d_{D^T}(v, w) \leq d_{D^T}(v, u) + d_{D^T}(u, w)$, and hence, we have $d_{D^T}(v, w) = d_{D^T}(v, u) + d_{D^T}(u, w)$. Thus, $d_{D^T}(u, w) < d_{D^T}(v, w)$, or equivalently, $d_D(w, u) < d_D(w, v)$. If $w \in \partial(D^T)$, we are done. If not, then for each $t \in V(D^T)$, there exists a vertex $w' \in N^+(w)$ such that $d_{D^T}(t, w') > d_{D^T}(t, w)$. In particular, we have $d_{D^T}(u, w') > d_{D^T}(u, w)$ and $d_{D^T}(v, w') > d_{D^T}(v, w)$. By applying the same arguments as above, we can show that $d_{D^T}(v, w') = d_{D^T}(v, w) + d_{D^T}(w, w')$ and $d_{D^T}(u, w') = d_{D^T}(u, w) + d_{D^T}(w, w')$. But since $d_{D^T}(v, w) = d_{D^T}(v, u) + d_{D^T}(u, w)$, we can write $d_{D^T}(v, w') = d_{D^T}(v, w) + d_{D^T}(w, w') = d_{D^T}(v, u) + d_{D^T}(u, w) + d_{D^T}(w, w') = d_{D^T}(v, u) + d_{D^T}(u, w')$. This implies that $d_{D^T}(u, w') < d_{D^T}(v, w')$, that is, $d_D(w', u) < d_D(w', v)$. Again, if $w' \in \partial(D^T)$, then we are done. Otherwise, we continue this process, and since $D$ is a finite graph and $\partial(D^T)$ is always non-empty, we eventually find a vertex $s \in \partial(D^T)$ such that $d_D(s, u) < d_D(s, v)$. Thus, $\partial(D^T)$ is a solid-resolving set of $D$.

\smallskip
2. Let $u$ be a forced vertex of $D$. By Theorem~\ref{thm:forced-characterization}, there exists a vertex $v \neq u$ such that $N^-(u) \subseteq N^-[v]$ in $D$. If $w \in N^-(u)$, then $w \in N^-[v]$ and hence, $d_D(w, v) \leq 1$ for all $w \in N^-(u)$. However, $d_D(u, v) \geq 1$ since $u$ and $v$ are distinct vertices. Hence, $d_D(w, v) \leq d_D(u, v)$ for all $w \in N^-(u)$ in $D$. This implies that $d_{D^T}(v, w) \leq d_{D^T}(v, u)$ for all $w \in N^+(u)$ in $D^T$. By definition of the boundary, this shows that $u \in \partial(D^T)$.

\smallskip
3. Consider the oriented graph $D$ in Figure~\ref{fig:boundarynotsubset}. From the figure, we have $\partial(D^T) = \{1, 2, 4, 5, 6\}$. Next, we will show that $\beta_S(D) = 5$. Observe that $\deg^-(1) = \deg^-(2) = \deg^-(6) = 1$ and $N^-(4) \subseteq N^-[1]$. Hence, by Corollary~\ref{cor:forceddegleq1} and Theorem~\ref{thm:forced-characterization}, we have that $1, 2, 4$ and $6$ are forced vertices. However, the set $S' = \{1, 2, 4, 6\}$ is not a solid-resolving set since there does not exist any $s \in S'$ such that $d(s, 3) < d(s, 4)$. We claim that in fact $S = \{1, 2, 3, 4, 6\}$ is a solid-resolving set of $D$. Observe that forced vertices distinguish themselves, and $d(1, 5) < d(1, 3) < d(1, 2) < d(1, 6)$ and $d(2, 5) < d(2, 1) = d(2, 4)$. Hence, by Theorem~\ref{thm:srsmain}, $S$ is a solid-resolving set of $D$. However, since $3 \notin \partial(D^T)$, this minimum solid-resolving set is not contained in the boundary of the transpose.
\end{proof}

\begin{figure}
		\centering
		
		\begin{tikzpicture}[>=Stealth, node distance=2cm,
			unfilled/.style={circle, draw, minimum size=2mm, inner sep=1pt},
			filled/.style={circle, draw, fill=black, minimum size=2mm, inner sep=1pt},
			labelstyle/.style={font=\small}
			]
			
			
			\begin{scope}[xshift=-3.5cm]
				
				\node (1) at (0,2)   {};
				\node (2) at (2,0)  {};
				\node (3) at (1,1)  {};
				\node (4) at (-1,1)  {};
				\node (5) at (2,2)  {};
				\node (6) at (0,0)   {};
				
				\foreach \n in {1,2,3,4,6} {
					\node[filled] at (\n) {};
				}
				\node[unfilled] at (5) {};
				
				\node[labelstyle, above]      at (1) {1};
				\node[labelstyle, below] 	  at (2) {2};
				\node[labelstyle, above left] at (3) {3};
				\node[labelstyle, left]       at (4) {4};
				\node[labelstyle, above]	  at (5) {5};
				\node[labelstyle, below] 	  at (6) {6};
				
				\draw[->, thick] (1) -- (4);
				\draw[->, thick] (1) -- (5);
				\draw[->, thick] (6) -- (1);
				\draw[->, thick] (3) -- (2);
				\draw[->, thick] (2) -- (5);
				\draw[->, thick] (2) -- (6);
				\draw[->, thick] (4) -- (3);
				\draw[->, thick] (5) -- (3);
				\draw[->, thick] (6) -- (3);
				\draw[->, thick] (6) -- (4);
				
				\node at (.5,-1) {$D$};
				
			\end{scope}

			
			\begin{scope}[xshift=3.5cm]
				
				\node (1) at (0,2)   {};
				\node (2) at (2,0)  {};
				\node (3) at (1,1)  {};
				\node (4) at (-1,1)  {};
				\node (5) at (2,2)  {};
				\node (6) at (0,0)   {};
				
				\foreach \n in {1,2,4,5,6} {
					\node[filled] at (\n) {};
				}
				\node[unfilled] at (3) {};
				
				\node[labelstyle, above]      at (1) {1};
				\node[labelstyle, below] 	  at (2) {2};
				\node[labelstyle, above left] at (3) {3};
				\node[labelstyle, left]       at (4) {4};
				\node[labelstyle, above]	  at (5) {5};
				\node[labelstyle, below] 	  at (6) {6};
				
				\draw[<-, thick] (1) -- (4);
				\draw[<-, thick] (1) -- (5);
				\draw[<-, thick] (6) -- (1);
				\draw[<-, thick] (3) -- (2);
				\draw[<-, thick] (2) -- (5);
				\draw[<-, thick] (2) -- (6);
				\draw[<-, thick] (4) -- (3);
				\draw[<-, thick] (5) -- (3);
				\draw[<-, thick] (6) -- (3);
				\draw[<-, thick] (6) -- (4);
				
				\node at (.5,-1) {$D^T$};
				
			\end{scope}
			
		\end{tikzpicture}
		
		\caption{An oriented graph $D$ where a minimum solid-resolving set is not contained in the boundary of $D^T$. The filled vertices in $D$ are a minimum solid-resolving set, and the filled vertices in $D^T$ are its boundary vertices.}
		\label{fig:boundarynotsubset}
	\end{figure}
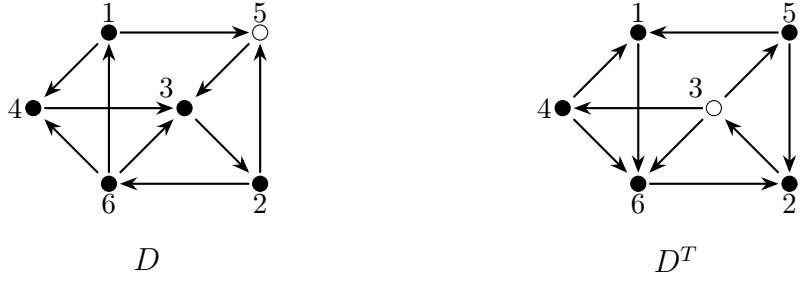

\section{Graph Families}\label{sec-graphfamilies}
We utilize the characterizations of solid-resolving sets and forced vertices established in the previous section to study the solid-metric dimension of digraphs whose underlying graphs are trees and orientations of complete graphs, cycles, and complete bipartite graphs.

\subsection{Complete graphs}
We first turn our attention to the solid-metric dimension of tournaments. Let $\overrightarrow{K_n}$ be an orientation of the complete graph on $n$ vertices.
\begin{observation}
\label{obs:completedegleq2}
    If $\deg^-(u) \leq 2$ for some vertex $u \in V(\overrightarrow{K_n})$, then $u$ is a forced vertex of~$\overrightarrow{K_n}$.
\end{observation}
\begin{proof}
    If $\deg^-(u) \leq 1$, then by Corollary~\ref{cor:forceddegleq1}, $u$ is a forced vertex. Let $N^-(u) = \{v, w\}$, and without loss of generality, let $\overrightarrow{wv} \in A(\overrightarrow{K_n})$. Then $N^-(u) \subseteq N^-[v]$ and hence, $u$ is a forced vertex.
\end{proof}

For complete graphs with a small number of vertices, it is possible to completely determine the possible solid-metric dimensions of their orientations. 
\begin{proposition}
    We have $\beta_S(\overrightarrow{K_3}) = 3$, $\beta_S(\overrightarrow{K_4}) \geq 3$, and $ \beta_S(\overrightarrow{K_5}) \geq 4$.
\end{proposition}
\begin{proof}
    No vertex of $K_3$ has degree greater than 2, and hence, by Observation~\ref{obs:completedegleq2}, all vertices of $\overrightarrow{K_3}$ are forced. That is, $\beta_S(\overrightarrow{K_3}) = 3$. 

    Let us assume that there exists an orientation $\overrightarrow{K_4}$ of $K_4$ such that $\beta_S(\overrightarrow{K_4}) \leq 2$. Let $V(\overrightarrow{K_4}) = \{v_1, v_2, v_3, v_4\}$, and without loss of generality, let $v_1$ and $v_2$ be two vertices which are not in some solid-resolving set of $\overrightarrow{K_4}$. By Observation~\ref{obs:completedegleq2}, we have that $3 \leq \deg^-(v_1), \deg^-(v_2) \leq 3$, that is $N^-(v_1) = \{v_2, v_3, v_4\}$ and $N^-(v_2) = \{v_1, v_3, v_4\}$. But this is a contradiction since $\overrightarrow{v_1 v_2}, \overrightarrow{v_2 v_1} \in A(\overrightarrow{K_4})$, while $\overrightarrow{K_4}$ is an orientation. Hence, $\beta_S(\overrightarrow{K_4}) \geq 3$.

    Let us assume that there exists an orientation $\overrightarrow{K_5}$ of $K_5$ such that $\beta_S(\overrightarrow{K_5}) \leq 3$. Let $V(\overrightarrow{K_5}) = \{v_1, v_2, v_3, v_4, v_5\}$, and without loss of generality, let $v_1$ and $v_2$ be two vertices which are not in some solid-resolving set of $\overrightarrow{K_5}$. By Observation~\ref{obs:completedegleq2}, we have that $3 \leq \deg^-(v_1), \deg^-(v_2) \leq 4$. If $\deg^-(v_1) = 4$, then $N^-(v_i) \subseteq N^-[v_1]$ for $2 \leq i \leq 5$. By Theorem~\ref{thm:forced-characterization}, this implies that  $\beta_S(\overrightarrow{K_5}) \geq 4$, a contradiction. Hence, $\deg^-(v_1) = \deg^-(v_2) = 3$. If $v_2 \in N^-(v_1)$, then without loss of generality, we must have $N^-(v_1) = \{v_2, v_3, v_4\}$ and $N^-(v_2) = \{v_3, v_4, v_5\}$. But now $N^-(v_1) \subset N^-[v_2]$ and by Theorem~\ref{thm:forced-characterization}, $v_1$ is a forced vertex, a contradiction. Hence, $v_2 \notin N^-(v_1)$. Now if $v_1 \notin N^-(v_2)$, then $N^-(v_1) = N^-(v_2)$, which leads to a contradiction by Theorem~\ref{thm:forced-characterization} as before. Therefore, without loss of generality, we must have $N^-(v_1) = \{v_3, v_4, v_5\}$ and $N^-(v_2) = \{v_1, v_3, v_4\}$. But then $N^-(v_2) \subset N^-[v_1]$, which leads to a contradiction by Theorem~\ref{thm:forced-characterization}. Hence, $\beta_S(\overrightarrow{K_5}) \geq 4$.
\end{proof}

While the maximum possible solid-metric dimension that an orientation of $K_n$ can attain can be shown by the following result to be equal to $n$, it is not trivial to determine the minimum possible solid-metric dimension. However, randomized computer searches have given us the bounds in Conjecture~\ref{conj:compgraphs}, which we believe are true.

\begin{proposition}
\label{prop:compub}
    There exists an orientation $\overrightarrow{K_n}$ such that $\beta_S(\overrightarrow{K_n}) = n$.
\end{proposition}
\begin{proof}
    Let $\overrightarrow{K_n}$ be the oriented graph with $V(\overrightarrow{K_n}) = \{v_1, v_2, \dots, v_n\}$ and $A(\overrightarrow{K_n}) = \{\overrightarrow{v_j v_i} \colon j > i\}$. Observe that $N^-(v_i) \subseteq N^-[v_{i + 1}]$, for $i \in \{1, \dots, n - 1\}$. Hence, by Theorem~\ref{thm:forced-characterization}, the vertices $\{v_1, \dots, v_{n - 1}\}$ are forced vertices of a solid-resolving set of $\overrightarrow{K_n}$. Moreover, since $\deg^-(v_n) = 0$, by Corollary~\ref{cor:forceddegleq1}, $v_n$ is a forced vertex as well. Hence, we have $\beta_S(\overrightarrow{K_n}) = n$.
\end{proof}

\begin{conjecture}
\label{conj:compgraphs}
    We have $\beta_S(\overrightarrow{K_6}) \geq 4$, $\beta_S(\overrightarrow{K_7}) \geq 5$, $ \beta_S(\overrightarrow{K_8}) \geq 5$, $\beta_S(\overrightarrow{K_9}) \geq 5$, $\beta_S(\overrightarrow{K_{10}}) \geq 5$ and $\beta_S(\overrightarrow{K_{11}}) \geq 6$.
\end{conjecture}

We have also observed during our computer searches that there exist orientations of $K_n$ which attain every possible solid-metric dimension between our conjectured lower bounds and the upper bound given by Proposition~\ref{prop:compub}. Determining a lower bound for arbitrary $n$ and proving that our observation holds in each case, would be an interesting topic for future research.

\subsection{Trees}
We next turn our attention to digraphs whose underlying undirected graphs are trees, and determine their solid-metric dimensions in the following theorem. We denote by $n_i^-(D)$ the number of vertices $u$ with $\deg^- (u) = i$ in $D$.

\begin{theorem}
\label{thm:treessmd}
		Let $D$ be a finite digraph whose underlying graph is a tree $T$ with $n$ vertices. The set of forced vertices is a solid-metric basis of $D$ and $\beta_S (D) = n^-_0 (D) + n^-_1 (D)$.
	\end{theorem}
	\begin{proof}
		Let $S$ be the set of forced vertices of $D$. Since the underlying undirected graph of~$D$ is a tree, a vertex $v$ is a forced vertex in $D$ if and only if $\deg^-(v) \leq 1$. In particular, if $v$ is a leaf in the underlying tree $T$, then $v \in S$. Thus, $|S| = n^-_1(D) + n^-_0(D)$, and it suffices to prove that $S$ is a solid-resolving set.
		
		In order to prove that $S$ is a solid-metric basis of $D$, it is sufficient to consider distinct~$u$ and~$v$ such that $u \notin S$, and prove that $d(s,u) < d(s,v)$ for some $s \in S$. Since $u \notin S$, there are at least two distinct in-neighbors of $u$, say $w_1$ and $w_2$. Consider briefly the undirected tree $T$. One of $w_1$ and $w_2$ may be on a path from $u$ to $v$, but at least one of them is not, say $w_1$. Thus, the path from $w_1$ to $v$ in $T$ contains~$u$. Since $w_1$ is an in-neighbor of $u$ in $D$, if there exists a directed path from $w_1$ to $v$ then it goes through~$u$. Thus, $d_D(w_1,u) < d_D(w_1,v)$. If no such directed path exists, then $d_D(w_1,v)=\infty$, and $d_D(w_1,u) < d_D(w_1,v)$. If $w_1 \in S$, then we are done. Otherwise, $w_1$ has at least two in-neighbors. Since the underlying graph of $D$ is a tree, if there exist directed paths from the in-neighbors of $w_1$ to $v$, then those paths go through $u$. Thus, when we continue this backtracking, we only meet vertices $w$ for which $d_D(w,u) < d_D (w,v)$ holds. Eventually, we will find some element of $S$, since at least the leaves of $T$ are forced vertices in $D$ and~$T$ is finite.
	\end{proof}
    
The previous result allows us to determine all the possible values of the solid-metric dimension for one of the simplest types of trees - paths. 
\begin{corollary}
\label{cor:pathsmd}
    Let $\overrightarrow{P_n}$ be a digraph whose underlying graph is the path on $n$ vertices $P_n$. We have $2 \leq \beta_S(\overrightarrow{P_n}) \leq n$.
\end{corollary}
\begin{proof}
    Since a path is also a tree, by Theorem~\ref{thm:treessmd}, we have $\beta_S(\overrightarrow{P_n}) = n^-_0(\overrightarrow{P_n}) + n^-_1(\overrightarrow{P_n}) = n - n^-_2(\overrightarrow{P_n})$. Hence, the extremal values of $\beta_S(\overrightarrow{P_n})$ occur when $n^-_2(\overrightarrow{P_n})$ is minimized or maximized. We can have $n^-_2(\overrightarrow{P_n}) = 0$ when $\overrightarrow{P_n}$ is a simple orientation of the undirected path from the initial vertex to the terminal vertex. We can have $n^-_2(\overrightarrow{P_n}) = n - 2$ (all vertices except the initial and terminal vertices) when all the edges are bidirected, with arcs in both directions.
\end{proof}

By Theorem~\ref{thm:treessmd} and Corollary~\ref{cor:pathsmd}, it is clear that the orientation of $P_n$ that has solid-metric dimension $k$, where $2 \leq k \leq n$, is that which has $n - k$ vertices with in-degree 2. Thus, we can construct directed paths $\overrightarrow{P_n}$ of any desired solid-metric dimension between 2 and $n$.

\subsection{Cycles}
Next, we consider orientations of cycles. As with trees, we can determine their solid-metric dimensions by determining their forced vertices, as demonstrated by the following result.

\begin{theorem}
\label{thm:cyclessmd}
    Let $\overrightarrow{C_n}$ be an orientation of the cycle with $n$ edges. If $n = 3$ or if $n = 4$ and $\overrightarrow{C_n}$ has only sources and sinks, then $\beta_S(\overrightarrow{C_n}) = n$, otherwise, $\beta_S(\overrightarrow{C_n}) = n - n_2^-(\overrightarrow{C_n})$.
\end{theorem}
\begin{proof}
    By Corollary~\ref{cor:forceddegleq1}, the vertices in the set $S = \{u \in V(\overrightarrow{C_n}) \mid \deg^-(u) \leq 1\}$ are forced vertices. Since, $|S| = n^-_1(\overrightarrow{C_n}) + n^-_0(\overrightarrow{C_n}) = n - n^-_2(\overrightarrow{C_n})$, it suffices to prove that $S$ is a solid-resolving set.

    Observe that in $\overrightarrow{C_n}$, sources and sinks occur in pairs. If $\overrightarrow{C_n}$ does not have any sources or sinks, then $d^-(u) = 1$ for all $u \in V(\overrightarrow{C_n})$. Hence, all vertices of $\overrightarrow{C_n}$ are forced, and the result is true. 
    
    Let us assume that $\overrightarrow{C_n}$ has at least one sink (or at least one source). For any distinct $s \in S$ and $t \in V(\overrightarrow{C_n}) \setminus \{s\}$ we trivially have $d(s, s) < d(s, t)$. Thus, it remains to distinguish any vertex $u \in V(\overrightarrow{C_n}) \setminus S$, that is, any vertex $u$ with $\deg^-(u) = 2$, from every other vertex in $\overrightarrow{C_n}$. Let $N^-(u) = \{p_1, p_2\}$. Clearly, $p_1, p_2 \in S$. Now consider any vertex $x \neq u$. Depending on the distances between $p_1$, $p_2$ and $x$, we have the following possibilities:
    \begin{itemize}
        \item If either $d(p_1, x) > 1$ or $d(p_2, x) > 1$, then we have that either $d(p_1, x) > d(p_1, u)$ or $d(p_2, x) > d(p_2, u)$. Hence, either $p_1$ or $p_2$ distinguishes $u$ from every other vertex $x$, and $S$ is a solid-resolving set.
        \item If $d(p_1, x) = 0$ and $d(p_2, x) = 1$, then $n = 3$ and we have either $N^-(u) \subseteq N^-[p_1]$ or $N^-(u) \subseteq N^-[p_2]$. In either case, by Theorem~\ref{thm:forced-characterization}, $u$ is also forced, and hence $\beta_S(\overrightarrow{C_3}) = 3$. We get a similar result if $d(p_1, x) = 1$ and $d(p_2, x) = 0$.
        \item If $d(p_1, x) = d(p_2, x) = 1$, then $n = 4$ and we have $N^-(u) = N^-(x)$. Hence, by Theorem~\ref{thm:forced-characterization}, $u$ and $x$ are also forced, and therefore $\beta_S(\overrightarrow{C_4}) = 4$ when $\overrightarrow{C_4}$ has two sources and two sinks. \qedhere
    \end{itemize}
\end{proof}

Theorem~\ref{thm:cyclessmd} precisely determines the solid-metric dimension of any oriented cycle. Using this result, the following corollary is quite straightforward.
\begin{corollary}
\label{cor:cyclessmd}
    Let $\overrightarrow{C_n}$ be an orientation of the cycle with $n$ edges. If $n \geq 5$, then $\lceil \frac{n}{2} \rceil \leq \beta_S(\overrightarrow{C_n}) \leq n$.
\end{corollary}
\begin{proof}
    The upper bound can be obtained by considering the unidirectional oriented cycle (which has no sources or sinks). Since in this case, $n_2^-(\overrightarrow{C_n}) = 0$, we have from Theorem~\ref{thm:cyclessmd} that $\beta_S(\overrightarrow{C_n}) = n$. 

    To prove the lower bound, consider the oriented cycle $\overrightarrow{C_n}$ where consecutive arcs are oriented in opposite directions. In this graph, every vertex is either a source or a sink (except maybe one vertex), that is $n_2^-(\overrightarrow{C_n}) = \lfloor \frac{n}{2} \rfloor = n_0^-(\overrightarrow{C_n})$. Hence, from Theorem~\ref{thm:cyclessmd}, we have $\beta_S(\overrightarrow{C_n}) = \lceil \frac{n}{2} \rceil$.
\end{proof}

Similar to the case of directed paths, in general, the solid-metric dimension of oriented cycles depends mainly on the number of sinks in the graph. By Theorem~\ref{thm:cyclessmd} and Corollary~\ref{cor:cyclessmd}, it is clear that the orientation of $C_n$ that has any solid-metric dimension $k$, where $\lceil \frac{n}{2} \rceil \leq k \leq n$, is that which has $n - k$ sinks. 

\subsection{Complete Bipartite Graphs}
Finally, we consider the family of orientations of complete bipartite graphs, where contrary to the observations so far in this section, the set of forced vertices solely may not be a solid-resolving set.

\begin{theorem}
    Let $D$ be an orientation of a complete bipartite graph $K_{m, n}$ and let $F \subseteq V(D)$ be the set of forced vertices, if any exist. If $D - F$ is an orientation of $K_{m', n'}$, where $1 \leq m' \leq m$ and $1 \leq n' \leq n$, then $\beta_S(D) \leq |F| + \min(m', n')$.  
\end{theorem}
\begin{proof}
    Let $D$ be an orientation of a complete bipartite graph $K_{m, n}$. Let $A$ and $B$ denote the sets of $m$ and $n$ vertices in the first and second partitions of $K_{m, n}$, respectively, and let $F \subseteq V(D)$ be the set of forced vertices. We have that $D - F$ is an orientation of $K_{m', n'}$, and without loss of generality, assume that $m' \leq n'$. Let $A'$ and $B'$ denote the sets of $m'$ and $n'$ vertices in the first and second partitions of $K_{m', n'}$, respectively. To prove the theorem, we prove that $S = F \cup A'$ is a solid-resolving set.

    For any distinct vertices $s \in S$ and $t \in V(D) \setminus \{s\}$, we trivially have $d(s, s) < d(s, t)$. Thus, it remains to distinguish any vertex $w \in V(D) \setminus S$, that is, any vertex $w \in B'$, from every other vertex in $D$. Since $w \notin S$, $\textrm{deg}^-(w) \geq 2$ and there exists a vertex $x \in A$ such that $xw$ is an arc in $D$. For any vertex $w' \neq x \in A'$, we then have $d(x, w) < d(x, w')$, that is, $x$ distinguishes $w$ from every vertex except $x$ in $A$. Since $\textrm{deg}^-(w) \geq 2$, there exists another vertex $y \in A$ such that $yw$ is an arc in $D$. Since, $d(y, w) < d(y, x)$, $y$ distinguishes $w$ from $x$. Now let $w'\in B$, with $w'\neq w$. Since $w \notin F$, we must have $N^-(w) \nsubseteq N^-[w']$, that is, there exists a vertex $x \in N^-(w) \setminus N^-[w']$. It follows that $d(x, w) < d(x, w')$. Hence, for every such vertex $w' \in B$, there exists a vertex $x$ that distinguishes $w$ from $w'$, and therefore $S$ is a solid-resolving set, and $\beta_S(D) \leq |F| + \min (|A'|, |B'|) = |F| + \min (m', n')$. 
\end{proof}

While a lower bound for the solid-metric dimension of orientations of complete bipartite graphs $\beta_S(D) \geq |F|$ is straightforward by the definition of forced vertices, it is easy to see that this bound is not tight. It would be interesting to determine tighter upper bounds, perhaps in terms of some other related graph parameter.

\section{Vertex and Arc Deletions}\label{sec-graphoperations}

The question of how vertex and edge deletions affect the metric dimension of a graph was first studied by Chartrand and Zhang in~\cite{ChartrandResolvabilitySurvey}. Eroh, Feit, Kang, and Yi~\cite{ErohEdgeDeletion} proved that the metric dimension of an undirected graph can decrease arbitrarily when one vertex or edge is removed. It can increase arbitrarily when one vertex is removed, however, the increase can be only 2 when one edge is removed. The same upper bound holds for the edge metric dimension, as shown by Wei, Yue, and Chen~\cite{WeiEMDEdgeDeletion}. However, for the connected metric dimension, the removal of one vertex or edge can result in an arbitrary increase or decrease according to Eroh, Kang, and Yi~\cite{ErohConnected}.
	
Motivated by these results, we consider how the solid-metric dimension of a digraph can change when we remove one vertex or arc. We prove that the decrease can be arbitrarily large when we remove one vertex or arc. The increase can also be arbitrarily large when we remove one vertex. However, when we remove one arc, the solid-metric dimension can increase by at most 1.

\subsection{Vertex Deletion}

We construct digraphs whose solid-metric dimension increases or decreases by an arbitrary integer~$k$ when one vertex is removed. These constructions demonstrate that the solid-metric dimension is not a local property but a global one.

\begin{proposition}\label{prop-vert-increase}
    For every integer $k \geq 1$, there exists a digraph $D_k$ with $|V(D_k)| = 2k + 4$ and a vertex $u \in V(D_k)$ such that $\beta_S(D_k - u) - \beta_S(D_k) = k$.
\end{proposition}
\begin{proof}
    Let $k \geq 1$ and $D_k$ be an oriented graph such that $V(D_k) = \{u, w\} \cup \{u_i, v_i \mid i \in [1, k + 1]\}$ and $A(D_k) = \{\overrightarrow{uw}\} \cup \{\overrightarrow{uu_i}, \overrightarrow{wu_i}, \overrightarrow{wv_i}, \overrightarrow{v_i u_i} \mid i \in [1, k + 1]\}$. The graph is illustrated in Figure~\ref{fig:vertexdeletionincrease}. We will show that in this graph $D_k$, when the vertex $u$ is removed, the solid-metric dimension increases by $k$.
    
    We first determine the solid-metric dimension of $D_k$. Observe first that since $\deg^-(v_i) = \deg^-(w) = 1$ and $\deg^-(u) = 0$, where $i \in [1, k + 1]$, by Corollary~\ref{cor:forceddegleq1}, $u, w$ and $v_i, i \in [1, k + 1]$ are forced vertices. Therefore, $\beta_S(D_k) \geq k + 3$. We claim that $S = \{u, w\} \cup \{v_i \mid i \in [1, k + 1]\}$ is a solid-resolving set and hence $\beta_S(D_k) = k + 3$. Observe that if $s$ is a forced vertex, then $d(s, s) < d(s, t)$ for any other vertex $t \in V(D_k)$. Thus, by Theorem~\ref{thm:srsmain}, to show that $S$ is a solid-resolving set, we only need to show that for all $u_i, t \in V(D_k), t \neq u_i$, there exists a vertex $s \in S$ such that $d(s, u_i) < d(s, t)$, where $i \in [1, k + 1]$. Since $d(v_i, x) = 1 \iff x = u_i$, we have $d(v_i, u_i) < d(v_i, t)$, for $t \neq v_i$, and since for any $x \neq u$, $d(u, x) = 1 \iff x \neq v_i$, we have $d(u, u_i) < d(u, v_j)$, where $i, j \in [1, k + 1]$. Hence, the claim is true.

    Next, we determine the solid-metric dimension of $D_k - u$. Observe that in $D_k - u$, since $\deg^-(v_i) = 1$ and $\deg^-(w) = 0$, where $i \in [1, k + 1]$, by Corollary~\ref{cor:forceddegleq1}, $w$ and $v_i, i \in [1, k + 1]$ are forced vertices. Moreover, since $N^-_{D_k - u}(u_i) \subseteq N^-_{D_k - u}[v_i], i \in [1, k + 1]$, by Theorem~\ref{thm:forced-characterization}, $u_i, i \in [1, k + 1]$ are also forced vertices in $D_k - u$. Hence, by definition, $\beta_S(D_k - u) = 2k + 3$. Thus, $\beta_S(D_k - u) - \beta_S(D_k) = k$.
\end{proof}

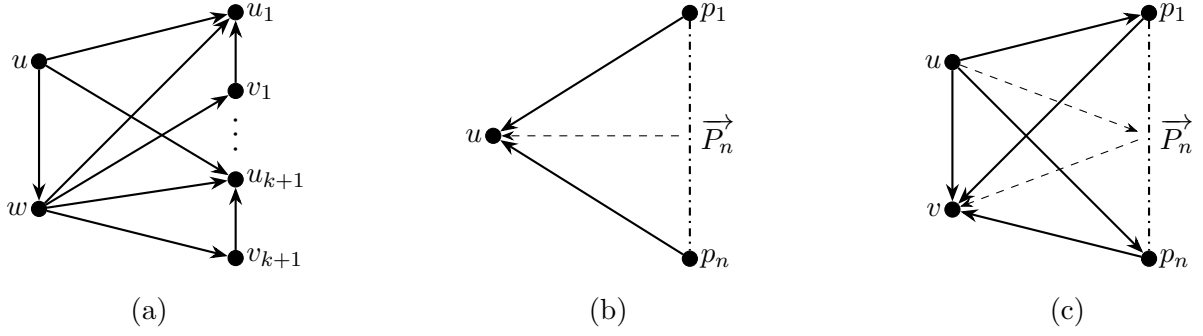
\begin{figure}
\centering
\begin{subfigure}{0.25\textwidth}
    \begin{tikzpicture}[>=Stealth, node distance=2cm,
    unfilled/.style={circle, draw, minimum size=2mm, inner sep=1pt},
    filled/.style={circle, draw, fill=black, minimum size=2mm, inner sep=1pt},
    labelstyle/.style={font=\small}, scale=0.65
]

\node[filled] (w) at (0,2)  {};
\node[filled] (u) at (0,5)  {};
\node[filled] (vk) at (4,1)  {};
\node[filled] (uk) at (4,2.6)  {};
\node[filled] (v1) at (4,4.4)  {};
\node[filled] (u1) at (4,6)  {};
\node[] (1) at (4,3.5)  {$\cdot$};
\node[] (2) at (4,3.2)  {$\cdot$};
\node[] (3) at (4,3.8)  {$\cdot$};

\node[labelstyle, left] at (w) {$w$};
\node[labelstyle, left] at (u) {$u$};
\node[labelstyle, right] at (vk) {$v_{k+1}$};
\node[labelstyle, right] at (uk) {$u_{k+1}$};
\node[labelstyle, right]at (v1) {$v_1$};
\node[labelstyle, right]at (u1) {$u_1$};

\draw[->, thick] (u) -- (w);
\draw[->, thick] (u) -- (u1);
\draw[->, thick] (u) -- (uk);
\draw[->, thick] (w) -- (u1);
\draw[->, thick] (w) -- (uk);
\draw[->, thick] (w) -- (v1);
\draw[->, thick] (w) -- (vk);
\draw[->, thick] (v1) -- (u1);
\draw[->, thick] (vk) -- (uk);

\end{tikzpicture}
    \caption{}
	\label{fig:vertexdeletionincrease}
\end{subfigure}
\hfill
\begin{subfigure}{0.25\textwidth}
    \begin{tikzpicture}[>=Stealth, node distance=2cm,
    unfilled/.style={circle, draw, minimum size=2mm, inner sep=1pt},
    filled/.style={circle, draw, fill=black, minimum size=2mm, inner sep=1pt},
    labelstyle/.style={font=\small}, scale=0.65
]

\node[filled] (u) at (0,3.5)  {};
\node[filled] (pn) at (4,1)  {};
\node[filled] (p1) at (4,6)  {};
\node[] (1) at (4,3.5)  {};

\node[labelstyle, left] at (u) {$u$};
\node[labelstyle, right] at (pn) {$p_n$};
\node[labelstyle, right] at (1) {$\overrightarrow{P_n}$};
\node[labelstyle, right]at (p1) {$p_1$};

\draw[thick, dash dot] (p1) -- (pn);
\draw[->, thick] (p1) -- (u);
\draw[->, thick] (pn) -- (u);
\draw[dashed, ->] (1) -- (u);

\end{tikzpicture}
    \caption{}
	\label{fig:vertexdeletiondecrease}
\end{subfigure}
\hfill
\begin{subfigure}{0.25\textwidth}
    \begin{tikzpicture}[>=Stealth, node distance=2cm,
    unfilled/.style={circle, draw, minimum size=2mm, inner sep=1pt},
    filled/.style={circle, draw, fill=black, minimum size=2mm, inner sep=1pt},
    labelstyle/.style={font=\small}, scale=0.65
]

\node[filled] (v) at (0,2)  {};
\node[filled] (u) at (0,5)  {};
\node[filled] (pn) at (4,1)  {};
\node[filled] (p1) at (4,6)  {};
\node[] (1) at (4,3.5)  {};

\node[labelstyle, left] at (v) {$v$};
\node[labelstyle, left] at (u) {$u$};
\node[labelstyle, right] at (pn) {$p_n$};
\node[labelstyle, right] at (1) {$\overrightarrow{P_n}$};
\node[labelstyle, right]at (p1) {$p_1$};

\draw[thick, dash dot] (p1) -- (pn);
\draw[->, thick] (u) -- (p1);
\draw[->, thick] (u) -- (pn);
\draw[dashed, ->] (u) -- (1);
\draw[->, thick] (u) -- (v);
\draw[->, thick] (p1) -- (v);
\draw[->, thick] (pn) -- (v);
\draw[dashed, ->] (1) -- (v);

\end{tikzpicture}
    \caption{}
	\label{fig:arcdeletiondecrease}
\end{subfigure}
\caption{Examples of graphs $D_k$ where removing a vertex or an arc causes arbitrary changes in the solid-metric dimension: (a) $\beta_S(D_k) = \beta_S(D_k - u) - k$, (b) $\beta_S(D_k - u) = \beta_S(D_k) - k$, and (c) $\beta_S(D_k - \overrightarrow{uv}) = \beta_S(D_k) - k$.}
\end{figure}

\begin{proposition}\label{prop-vert-decrease}
    For every integer $k \geq 1$, there exists a digraph $D_k$ with $|V(D_k)| \geq \max\{5, k + 3\}$ and 
    a vertex $u \in V(D_k)$ such that $\beta_S(D_k) - \beta_S(D_k - u) = k$.
\end{proposition}
\begin{proof}
    Let $k \geq 1, n \geq \max\{4, k+2\}$ and let $\overrightarrow{P_n}$ be a digraph whose underlying graph is the path on $n$ vertices $P_n$. Let $n^-_2(\overrightarrow{P_n}) = k$ and let $p_1$ and $p_n$ denote the initial and terminal vertices of $\overrightarrow{P_n}$. Construct the graph $D_k$ by adding a new vertex $u$ to $\overrightarrow{P_n}$ and arcs of the form $\overrightarrow{vu}$ for all $v \in V(\overrightarrow{P_n})$. The graph is illustrated in Figure~\ref{fig:vertexdeletiondecrease}. For every vertex $v\neq u$, we have $N^-(v) \subseteq N^-[u]$. Hence, by Theorem~\ref{thm:forced-characterization}, all the vertices in $S = V(D_k) \setminus \{u\}$ are forced vertices. Since $n \geq \max\{4, k+2\}$, for any vertex $x \in V(D_k) \setminus \{u\}$, the structure of $D_k$ ensures that either $d(p_1, u) < d(p_1, x)$ or $d(p_n, u) < d(p_n, x)$. Hence, $S$ is a solid-resolving set, and $\beta_S(D_k) = |S| = n$. On the other hand, since $D_k - u \cong \overrightarrow{P_n}$, by Theorem~\ref{thm:treessmd}, we have $\beta_S(D_k - u) = \beta_S(\overrightarrow{P_n}) = n - k$. Consequently, we have $\beta_S(D_k) - \beta_S(D_k - u) = k$.
\end{proof}

In Propositions~\ref{prop-vert-increase} and~\ref{prop-vert-decrease}, the increase or decrease of the solid-metric dimension depends on the change in the number of forced vertices. There exist digraphs where this is not the case, for example, see the digraphs in Figure~\ref{fig-vert-examples}. The digraph in Figure~\ref{fig-vert-increase} has solid-metric dimension 4, and four forced vertices. When vertex 7 and all incident arcs are deleted, the solid-metric dimension increases by~1. However, the forced vertices of the resulting digraph are the same four forced vertices as of the original digraph. Similarly, the solid-metric dimension of the digraph in Figure~\ref{fig-vert-decrease} is 4. It has three forced vertices which, together with either vertex 2 or 4, form a solid-metric basis. When vertex 6 and all incident arcs are removed, the solid-metric dimension decreases by 1, while the forced vertices are the same as those of the original digraph.

\begin{figure}[h]
	\centering
		\begin{subfigure}[b]{0.45\linewidth}
			\centering
			\begin{tikzpicture}[>=Stealth, node distance=2cm,
				unfilled/.style={circle, draw, minimum size=2mm, inner sep=1pt},
				filled/.style={circle, draw, fill=black, minimum size=2mm, inner sep=1pt},
				labelstyle/.style={font=\small}
				]
				
				\node (1) at (0,1.5)   {};
				\node (2) at (2,3)  {};
				\node (3) at (4,3)  {};
				\node (4) at (6,1.5)  {};
				\node (5) at (4,0)  {};
				\node (6) at (2,0)   {};
				\node (7) at (1.7,1.7)   {};
				
				\foreach \n in {1,2,3,4} {
					\node[filled] at (\n) {};
				}
				\node[unfilled] at (5) {};
				\node[unfilled] at (6) {};
				\node[unfilled] at (7) {};
				
				\node[labelstyle, left]      at (1) {1};
				\node[labelstyle, above] 	  at (2) {2};
				\node[labelstyle, above]	 at (3) {3};
				\node[labelstyle, right]       at (4) {4};
				\node[labelstyle, below]	  at (5) {5};
				\node[labelstyle, below] 	  at (6) {6};
				\node[labelstyle, above] 	  at (7) {7};
				
				\draw[->, thick] (1) -- (2);
				\draw[->, thick] (2) -- (3);
				\draw[->, thick] (3) -- (4);
				\draw[->, thick] (4) -- (5);
				\draw[->, thick] (5) -- (6);
				\draw[->, thick] (6) -- (1);
				\draw[->, thick] (2) -- (5);
				\draw[->, thick] (3) -- (6);
				\draw[->, dashed, thick] (1) -- (7);
				\draw[->, dashed, thick] (3) -- (7);
				\draw[->, dashed, thick] (7) -- (6);				
			\end{tikzpicture}
			\caption{}\label{fig-vert-increase}
		\end{subfigure}
		\hfill
		\begin{subfigure}[b]{0.45\linewidth}
			\centering
			\begin{tikzpicture}[>=Stealth, node distance=2cm,
				unfilled/.style={circle, draw, minimum size=2mm, inner sep=1pt},
				filled/.style={circle, draw, fill=black, minimum size=2mm, inner sep=1pt},
				labelstyle/.style={font=\small}
				]
				
				\node (1) at (0,0)  {};
				\node (2) at (2,0)  {};
				\node (3) at (-1,2)  {};
				\node (4) at (1,1.5)  {};
				\node (5) at (3,2)  {};
				\node (6) at (1,3)  {};
				
				\foreach \n in {1,3,5} {
					\node[filled] at (\n) {};
				}
				\foreach \n in {2,4,6} {
					\node[unfilled] at (\n) {};
				}
				
				\node[labelstyle, below]      at (1) {1};
				\node[labelstyle, below] 	  at (2) {2};
				\node[labelstyle, left]	 at (3) {3};
				\node[labelstyle, above]       at (4) {4};
				\node[labelstyle, right]	  at (5) {5};
				\node[labelstyle, above] 	  at (6) {6};
				
				\draw[->, thick] (1) -- (2);
				\draw[->, thick] (1) -- (3);
				\draw[->, thick] (4) -- (1);
				\draw[->, thick] (2) -- (4);
				\draw[->, thick] (5) -- (2);
				\draw[->, thick] (3) -- (4);
				\draw[->, dashed, thick] (3) -- (6);
				\draw[->, thick] (4) -- (5);
				\draw[->, dashed, thick] (5) -- (6);	
			\end{tikzpicture}
			\caption{}\label{fig-vert-decrease}
		\end{subfigure}
	\caption{Two digraphs (including the dashed arcs) with the forced vertices illustrated in black. When the vertex 7 of (a) or vertex 6 of (b), and their incident arcs are removed, the solid-metric dimension of (a) increases, and that of (b) decreases by~1, but the forced vertices remain the same as the original graph.}\label{fig-vert-examples}
\end{figure}
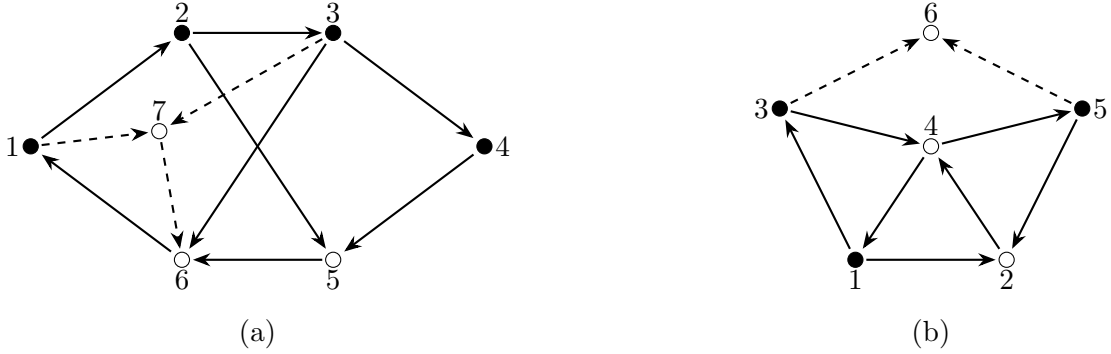

\subsection{Arc Deletion}

As in the previous subsection, in the following result, we construct a digraph whose solid-metric dimension decreases by an arbitrary integer~$k$ when one arc is removed.

\begin{proposition}\label{prop-arcdecrease}
    For every integer $k \geq 1$, there exists a digraph $D_k$ with $|V(D_k)| \geq \max\{6, k + 4\}$ and an arc $\overrightarrow{uv} \in A(D_k)$ such that $\beta_S(D_k) - \beta_S(D_k - \overrightarrow{uv}) = k$.
\end{proposition}
\begin{proof}
     Let $k \geq 1$, $n \geq \max\{4, k+2\}$, and let $\overrightarrow{P_n}$ be a digraph whose underlying graph is the path on $n$ vertices $P_n$. Let $n^-_2(\overrightarrow{P_n}) = k$ and let $p_1$ and $p_n$ denote the initial and terminal vertices of the path $\overrightarrow{P_n}$, respectively. Construct the graph $D_k$ by adding new vertices $u$ and $v$ to $\overrightarrow{P_n}$ and arcs of the form $\overrightarrow{uw}$ for all $w \in V(\overrightarrow{P_n})$ and arcs of the form $\overrightarrow{wv}$ for all $w \in V(\overrightarrow{P_n}) \cup \{u\}$. The graph is illustrated in Figure~\ref{fig:arcdeletiondecrease}. For every vertex $w\neq v$, we have $N^-(w) \subseteq N^-[v]$. Hence, by Theorem~\ref{thm:forced-characterization}, all the vertices in $S = V(D_k) \setminus \{v\}$ are forced vertices. For any vertex $x \in V(D_k) \setminus \{v\}$, the structure of $D_k$ ensures that either $d(p_1, v) < d(p_1, x)$ or $d(p_n, v) < d(p_n, x)$. Hence, $S$ is a solid-resolving set, and $\beta_S(D_k) = |S| = n + 1$. 
     
     Next, consider the graph $D_k - \overrightarrow{uv}$. Observe that $\deg^-_{D_k - \overrightarrow{uv}}(w) = \deg^-_{\overrightarrow{P_n}}(w) + 1$, for $w \in V(\overrightarrow{P_n})$. Hence, every vertex with in-degree $0$ in $\overrightarrow{P_n}$ has in-degree $1$ in $D_k - \overrightarrow{uv}$ and by Corollary~\ref{cor:forceddegleq1}, remains forced in $D_k - \overrightarrow{uv}$. On the other hand, if a vertex $w$ has in-degree $1$ in $\overrightarrow{P_n}$, with $\overrightarrow{w'w} \in A(\overrightarrow{P_n})$ for some vertex $w' \in V(\overrightarrow{P_n})$, then $w$ will have in-degree $2$ in $D_k - \overrightarrow{uv}$, but is still forced in $D_k - \overrightarrow{uv}$ by Theorem~\ref{thm:forced-characterization} as $N^-(w) \subseteq N^-[w']$. Since $\deg^-(u) = 0$ as well, if $S = \{w \in V(\overrightarrow{P_n}) \mid \deg^-_{\overrightarrow{P_n}}(w) \leq 1\} \cup \{u\}$, then we have $\beta_S(D_k - \overrightarrow{uv}) \geq |S|$. 
     
     We will then prove that $S$ is a solid-metric basis of $D_k - \overrightarrow{uv}$.  Notice that none of the directed paths between any two vertices in $\overrightarrow{P_n}$ visit $u$ or $v$. Thus, the distance from $p_i$ to $p_j$, where $p_i,p_j \in V(\overrightarrow{P_n})$, is the same in both $D_k - \overrightarrow{uv}$ and $\overrightarrow{P_n}$. Since $S \cap V(\overrightarrow{P_n})$ is precisely the set of forced vertices of $\overrightarrow{P_n}$, vertex pairs within $V(\overrightarrow{P_n})$ are distinguished according to Theorem~\ref{thm:treessmd}. Since $p_1,p_n \in S$, $n \geq 4$, and $d(p_1,v)=d(p_n,v)=1$, we have $d(p_1,v)<d(p_1,p_i)$ or $d(p_n,v)<d(p_n,p_i)$ for each $p_i \in V(\overrightarrow{P_n})$. On the other hand, we have $d(u,p_i)=1<d(u,v)=2$. Thus, $S$ is a solid-metric basis of $D_k - \overrightarrow{uv}$.
     Hence, we have $\beta_S(D_k - \overrightarrow{uv}) = \beta_S(\overrightarrow{P_n}) + 1 = n + 1 - k$. Consequently, we have $\beta_S(D_k) - \beta_S(D_k - \overrightarrow{uv}) = k$.
\end{proof}

Unlike in the case of vertex deletion, the following result shows that deleting an arc can increase the solid-metric dimension of a digraph by at most $1$.
\begin{proposition}\label{prop-arcincrease}
		Let $D$ be a digraph and $\overrightarrow{uv} \in A(D)$. If $S$ is a solid-resolving set of $D$, then $S \cup \{v\}$ is a solid-resolving set of $D - \overrightarrow{uv}$. Consequently, $\beta_S(D - \overrightarrow{uv}) -  \beta_S (D) \leq 1$.
	\end{proposition}
	\begin{proof}
		Suppose to the contrary that $S \cup \{v\}$ is not a solid-resolving set of $H=D - \overrightarrow{uv}$. Then there exist distinct vertices $x$ and $y$ such that $d_H(s,x) \geq d_H(s,y)$ for all $s \in S \cup \{v\}$. In particular, $d_H(v,x) \geq d_H(v,y)$. However, $S$ is a solid-resolving set of $D$, and thus there exists some $t \in S$ such that $d_D(t,x) < d_D(t,y)$. The removal of $\overrightarrow{uv}$ cannot make distances in $H$ shorter compared to the corresponding distances in $D$. Thus,  
		\[d_H(t,x) \geq d_H(t,y) \geq d_D (t,y) > d_D (t,x)\]
		and we have $d_H (t,x) > d_D (t,x)$. This implies that all shortest paths from $t$ to $x$ in $D$ include the arc $\overrightarrow{uv}$. Clearly, no shortest path from $v$ to $x$ or $y$ in $D$ includes the arc $\overrightarrow{uv}$ and consequently $d_H(v,x) = d_D(v,x)$ and $d_H(v,y) = d_D(v,y)$. Now,
		 \begin{align*}
		 	d_D (t,y) & \leq d_D (t,v) + d_D(v,y) \leq d_D (t,v) + d_H(v,x) \\
		 	& = d_D (t,v) + d_D(v,x) = d_D (t,x),
		 \end{align*}
		 a contradiction.
		 
		 Therefore, $S \cup \{v\}$ is a solid-resolving set of $H = D - \overrightarrow{uv}$. This implies that $\beta_S (D - \overrightarrow{uv}) \leq \beta_S (D) + 1$ (when $S$ is chosen to be a solid-metric basis of $D$).
	\end{proof}

    There exist digraphs whose solid-metric dimension increases by~1 when an arc is removed; for example, when the underlying graph is a cycle or a path, such digraphs are easy to construct. For these examples, the increase is due to the generation of another forced vertex. However, there exist digraphs where removing an arc increases the solid-metric dimension without generating new forced vertices. One such example is illustrated in Figure~\ref{fig-arc-increase}. The set $\{1,2,3,4\}$ is the only solid-metric basis of this digraph. When the arc $\overrightarrow{26}$ is removed, the forced vertices remain the same, but the solid-metric dimension increases by~1.

    The construction in Proposition~\ref{prop-arcdecrease} is also dependent on the change in the number of forced vertices. However, as with arc deletions, there exist digraphs where the decrease in the solid-metric dimension does not depend on forced vertices, for instance the digraph illustrated in Figure~\ref{fig-arc-decrease}. The solid-metric dimension of this digraph is 4. The three forced vertices together with either vertex 2 or 6 are a solid-metric basis of the digraph. When the arc $\overrightarrow{13}$ is removed, the solid-metric dimension decreases by~1, but the forced vertices are the same as that of the original digraph.

    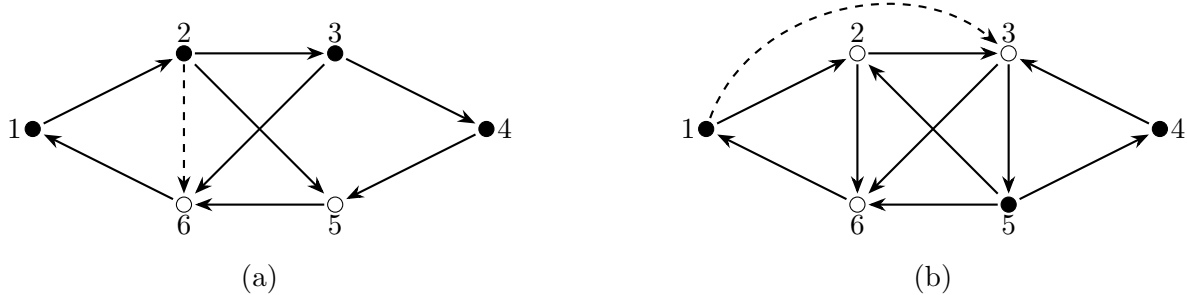
\begin{figure}[h]
	\centering
	\begin{subfigure}[b]{0.45\linewidth}
		\centering
		\begin{tikzpicture}[>=Stealth, node distance=2cm,
			unfilled/.style={circle, draw, minimum size=2mm, inner sep=1pt},
			filled/.style={circle, draw, fill=black, minimum size=2mm, inner sep=1pt},
			labelstyle/.style={font=\small}
			]
				
				\node (1) at (0,1)   {};
				\node (2) at (2,2)  {};
				\node (3) at (4,2)  {};
				\node (4) at (6,1)  {};
				\node (5) at (4,0)  {};
				\node (6) at (2,0)   {};
				
				\foreach \n in {1,2,3,4} {
					\node[filled] at (\n) {};
				}
				\node[unfilled] at (5) {};
				\node[unfilled] at (6) {};
				
				\node[labelstyle, left]      at (1) {1};
				\node[labelstyle, above] 	  at (2) {2};
				\node[labelstyle, above]	 at (3) {3};
				\node[labelstyle, right]       at (4) {4};
				\node[labelstyle, below]	  at (5) {5};
				\node[labelstyle, below] 	  at (6) {6};
				
				\draw[->, thick] (1) -- (2);
				\draw[->, thick] (2) -- (3);
				\draw[->, thick] (3) -- (4);
				\draw[->, thick] (4) -- (5);
				\draw[->, thick] (5) -- (6);
				\draw[->, thick] (6) -- (1);
				\draw[->, thick] (2) -- (5);
				\draw[->, thick] (3) -- (6);
				\draw[->, dashed, thick] (2) -- (6);				
		\end{tikzpicture}
		\caption{}\label{fig-arc-increase}
	\end{subfigure}
	\hfill
	\begin{subfigure}[b]{0.45\linewidth}
		\centering
		\begin{tikzpicture}[>=Stealth, node distance=2cm,
			unfilled/.style={circle, draw, minimum size=2mm, inner sep=1pt},
			filled/.style={circle, draw, fill=black, minimum size=2mm, inner sep=1pt},
			labelstyle/.style={font=\small}
			]
			
			\node (1) at (0,1)  {};
			\node (2) at (2,2)  {};
			\node (3) at (4,2)  {};
			\node (4) at (6,1)  {};
			\node (5) at (4,0)  {};
			\node (6) at (2,0)  {};
			
			\foreach \n in {1,4,5} {
				\node[filled] at (\n) {};
			}
			\foreach \n in {2,3,6} {
				\node[unfilled] at (\n) {};
			}
			
			\node[labelstyle, left]      at (1) {1};
			\node[labelstyle, above] 	  at (2) {2};
			\node[labelstyle, above]	 at (3) {3};
			\node[labelstyle, right]       at (4) {4};
			\node[labelstyle, below]	  at (5) {5};
			\node[labelstyle, below] 	  at (6) {6};
			
			\draw[->, thick] (1) -- (2);
			\draw[->, thick] (2) -- (3);
			\draw[->, thick] (4) -- (3);
			\draw[->, thick] (5) -- (4);
			\draw[->, thick] (5) -- (6);
			\draw[->, thick] (6) -- (1);
			\draw[->, thick] (2) -- (6);
			\draw[->, thick] (5) -- (2);
			\draw[->, thick] (3) -- (5);
			\draw[->, thick] (3) -- (6);
			\draw[->, dashed, thick] (1) to[out=70,in=140,looseness=1] (3);				
		\end{tikzpicture}
		\caption{}\label{fig-arc-decrease}
	\end{subfigure}
		\caption{Two digraphs (including the dashed arc(s)) with the forced vertices illustrated in black. When the corresponding dashed arc is removed, the solid-metric dimension of (a) increases, of (b) decreases by~1, but the forced vertices of both digraphs remain the same as the original digraphs'.}
	\end{figure}
    
	\begin{remark}
		Recall that when we remove one edge of an undirected graph, the increase in metric dimension and edge metric dimension can be 2. The difference between this behavior and the conclusion of Proposition~\ref{prop-arcincrease} seems to be due to directionality rather than the variant of metric dimension we are considering. Let us consider briefly solid-resolving sets in undirected graphs. We can attempt a similar technique as in the proof of Proposition~\ref{prop-arcincrease}. When $v$ is the end-vertex of the removed edge $uv$ that is closest to $x$ in all shortest paths from $t$ to $x$, we can indeed prove that $d_{G-uv}(v,x) < d_{G-uv}(v,y)$, and the pair $x,y$ would be distinguished by $S \cup \{v\}$. However, $x,y$ is not necessarily the only pair we need to consider. For those other pairs, $u$ might be the vertex we need to add to $S$. Thus, as is the case with the metric dimension of undirected graphs, we can prove that $S \cup \{u,v\}$ is a solid-resolving set of the undirected graph $G-uv$.
	\end{remark}

    To the best of our knowledge, the effects of vertex and arc deletions on the metric dimension of a digraph have not been studied. Thus, extending these results to the metric dimension of digraphs would be of interest.

\section{Relation with the Underlying Graph}
\label{sec-smdundirected}
Let $D$ be a digraph with the property that $\overrightarrow{uv} \in A(D)$ if and only if $\overrightarrow{vu} \in A(D)$. Denote its underlying undirected graph by $G$. Observe that since $d_D(u, v) = d_D(v, u)$ for any distinct vertices $u, v \in V(D)$, we have that $d_D(s, u) < d_D(s, v) \iff d_G(s, u) < d_G(s, v)$. That is, $S$ is a solid-metric basis in $D$ if and only if $S$ is a solid-metric basis in $G$. This tells us that the entire study of solid-metric basis and solid-metric dimension in undirected graphs can be considered to be completely contained in the study of solid-metric basis and solid-metric dimension in digraphs. This observation is useful for extending results for digraphs to undirected graphs. On the other hand, we show that there is no trivial relation between the solid-metric dimension of a digraph and the solid-metric dimension of its underlying undirected graph. To be precise, we show that there exist digraphs for which the absolute difference between their solid-metric dimension and that of their underlying undirected graphs can be arbitrarily large.

\begin{proposition}
    For every integer $k \geq 1$ and $n \geq k + 2$, there exists a digraph $D_k$ with $|V(D_k)| = n$ such that $\beta_S(D_k) - \beta_S(G_k) = k$, where $G_k$ is the underlying undirected graph of $D_k$.
\end{proposition}
\begin{proof}
    Let $k \geq 1$, $n \geq k + 2$ and let $D_k$ be a directed path $\overrightarrow{P_n}$ with the property that $n_2^-(\overrightarrow{P_n}) = n - k - 2$. Hence, $G_k = P_n$ is the underlying undirected path. By Theorem~\ref{thm:treessmd}, we have that $\beta_S(D_k) = n - (n - k - 2) = k + 2$. Moreover, from \cite{Solid}, we have that $\beta_S(G_k) = 2$. Hence, $\beta_S(D_k) - \beta_S(G_k) = k$, as required.
\end{proof}

\begin{figure}
\centering
\begin{tikzpicture}[>=Stealth, node distance=2cm,
    unfilled/.style={circle, draw, minimum size=2mm, inner sep=1pt},
    filled/.style={circle, draw, fill=black, minimum size=2mm, inner sep=1pt},
    labelstyle/.style={font=\small}
]

\node[filled] (u1) at (4.5, 2)  {};
\node[filled] (v1) at (4.5, 1)  {};
\node[filled] (w1) at (3, 0)  {};
\node[filled] (w2) at (6, 0)  {};
\node[] (1) at (6.75, 0)  {$\cdot$};
\node[] (2) at (7.5, 0)  {$\cdot$};
\node[] (3) at (8.25, 0)  {$\cdot$};
\node[filled] (ui) at (10.5, 2)  {};
\node[filled] (vi) at (10.5, 1)  {};
\node[filled] (wi) at (9, 0)  {};
\node[filled] (wi+1) at (12, 0)  {};
\node[] (4) at (12.75, 0)  {$\cdot$};
\node[] (5) at (13.5, 0)  {$\cdot$};
\node[] (6) at (14.25, 0)  {$\cdot$};
\node[filled] (uk) at (16.5, 2)  {};
\node[filled] (vk) at (16.5, 1)  {};
\node[filled] (wk) at (15, 0)  {};
\node[filled] (wk+1) at (18, 0)  {};

\node[labelstyle, below] at (w1) {$w_1$};
\node[labelstyle, below] at (w2) {$w_2$};
\node[labelstyle, below] at (v1) {$v_1$};
\node[labelstyle, above] at (u1) {$u_1$};
\node[labelstyle, below] at (wi) {$w_i$};
\node[labelstyle, below] at (wi+1) {$w_{i + 1}$};
\node[labelstyle, below] at (vi) {$v_i$};
\node[labelstyle, above] at (ui) {$u_i$};
\node[labelstyle, below] at (wk) {$w_k$};
\node[labelstyle, below] at (wk+1) {$w_{k + 1}$};
\node[labelstyle, below] at (vk) {$v_k$};
\node[labelstyle, above] at (uk) {$u_k$};

\draw[->, thick] (w1) -- (w2);
\draw[->, thick] (wk) -- (wk+1);
\draw[->, thick] (v1) -- (u1);
\draw[->, thick] (w1) -- (u1);
\draw[->, thick] (w2) -- (u1);
\draw[->, thick] (w2) -- (v1);
\draw[->, thick] (v1) -- (w1);
\draw[->, thick] (vk) -- (uk);
\draw[->, thick] (wk) -- (uk);
\draw[->, thick] (wk+1) -- (uk);
\draw[->, thick] (wk+1) -- (vk);
\draw[->, thick] (vk) -- (wk);
\draw[->, thick] (vi) -- (ui);
\draw[->, thick] (wi) -- (ui);
\draw[->, thick] (wi+1) -- (ui);
\draw[->, thick] (wi+1) -- (vi);
\draw[->, thick] (vi) -- (wi);
\draw[->, thick] (wi) -- (wi+1);
\end{tikzpicture}
    \caption{A graph $D_k$ with the property that $\beta_S(G_k) - \beta_S(D_k) = k$.}
	\label{fig:smdundirectedgreater}
\end{figure}
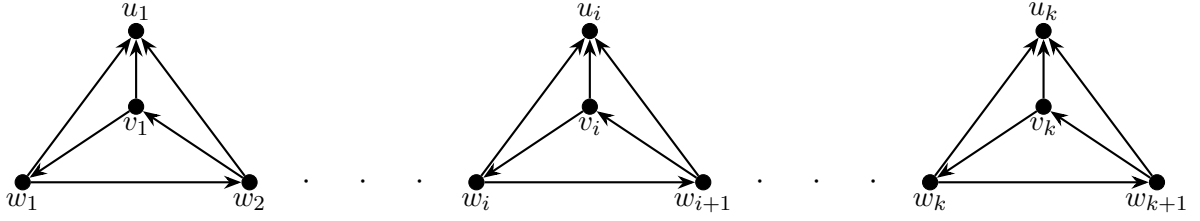

\begin{proposition}
    For every integer $k \geq 1$, there exists a digraph $D_k$ such that $\beta_S(G_k) - \beta_S(D_k) = k$, where $G_k$ is the underlying undirected graph of $D_k$.
\end{proposition}
\begin{proof}
\sloppy
    Let $D_k$ be the digraph obtained by taking $k$ copies of a particular orientation of~$K_4$ and joining them as shown in Figure~\ref{fig:smdundirectedgreater}. To be precise, let $V(D_k) = \{u_i, v_i, w_i \mid i \in [1, k]\} \cup \{w_{k + 1}\}$ and $A(D_k) = \{\overrightarrow{w_iu_i}, \overrightarrow{v_i u_i}, \overrightarrow{w_{i + 1} u_i}, \overrightarrow{w_i w_{i + 1}}, \overrightarrow{w_{i + 1} v_i}, \overrightarrow{v_i w_i} \mid i \in [1, k]\}$. We shall determine the solid-metric dimension of both $D_k$ as well as its underlying undirected graph~$G_k$ and show that it exhibits the desired property.

    First observe that the vertices in the set $S = \{v_i \mid i \in [1, k]\} \cup \{w_1, w_{k + 1}\}$ all have in-degree 1, and hence, by Corollary~\ref{cor:forceddegleq1}, are forced. We claim that $S$ is a solid-resolving set of $D_k$. For any distinct $s \in S$ and $t \in V(D_k) \setminus \{s\}$, we trivially have $d(s, s) < d(s, t)$. Thus, it remains to distinguish any vertex $x \in V(D_k) \setminus S$ from every other vertex in $D_k$. Observe that for $i \in [1, k]$, $d(v_i, u_i) < d(v_i, v)$, where $v$ can be $w_\ell, u_j$ or $v_j$, for $\ell \in [1, k + 1]$, $\ell \neq i$ and $j \in [1, k]$, $j \neq i$. Furthermore, $d(v_{i + 1}, u_i) < d(v_{i + 1}, w_i)$ and $d(v_{j - 1}, u_j) < d(v_{j - 1}, v_j)$, for $i \in [1, k - 1], j \in [2, k]$, and $d(w_{k + 1}, u_k) < d(w_{k + 1}, w_k)$ and $d(w_1, u_1) < d(w_1, v_1)$. Hence, all vertices $u_i, i \in [1, k]$ are distinguished from every other vertex in $D_k$ by $S$. Similarly, we have that either $d(v_{i - 1}, w_i) < d(v_{i - 1}, x)$ or $d(v_i, w_i) < d(v_i, x)$, where $i \in [2, k - 1]$, will always be true depending on the chosen vertex $x, x \neq w_i$. Hence, all vertices in $V(D_k) \setminus S$ are distinguished, and $S$ is a solid-resolving set. Hence, $\beta_S(D_k) = k + 2$.

    Next observe that all the vertices in the set $S' = \{u_i, v_i \mid i \in [1, k]\} \cup \{w_1, w_{k + 1}\}$ are forced vertices of $G_k$. We claim that $S'$ is a solid-resolving set of $G_k$. Just as before, it is enough to show that the vertices in $V(G_k) \setminus S'$ are distinguished from the other vertices in $G_k$. Clearly, either $d(v_i , w_i) < d(v_i, x)$ or $d(v_{i - 1} , w_i) < d(v_{i - 1}, x)$ where $i \in [2, k - 1]$, is always true depending on the choice of the vertex $x, x \neq w_i$. Hence, all vertices in $V(G_k) \setminus S'$ are distinguished, and $S'$ is a solid-resolving set. Hence, $\beta_S(G_k) = 2k + 2$, and therefore, $\beta_S(G_k) - \beta_S(D_k) = k$, as required.
\end{proof}

Several related graph parameters have been studied, including metric dimension, strong metric dimension, fractional metric dimension, etc. It would be interesting to find examples of graphs which show that there are no trivial relations between any of these parameters, akin to the two results above.

\section{Cartesian Products of Orientations}\label{sec-cartesianproducts}
    In this section, we determine the solid-metric dimension of a Cartesian product consisting of oriented paths and cycles (with certain restrictions). As was seen in Section~\ref{sec-graphfamilies}, the solid-metric dimension of oriented paths and cycles is determined by the number of forced vertices. We will prove that this is true also for their Cartesian products. To that end, we will first consider forced vertices in Cartesian products.
    
	Let $D_1$ and $D_2$ be oriented graphs. Their \emph{Cartesian product} $D_1 \Box D_2$ is the digraph with
	\begin{align*}
		V(D_1 \Box D_2) & = V(D_1) \times V(D_2) = \{uv \mid u \in V(D_1), v \in V(D_2)\} \\
		A(D_1 \Box D_2) & = \{ (u_1v_1,u_2v_2) \mid (u_1,u_2) \in A(D_1) \text{ and } v_1=v_2 \text{, or } \\ & \qquad \qquad \qquad u_1=u_2 \text{ and } (v_1,v_2) \in A(D_2) \}
	\end{align*}
	where $\times$ is the Cartesian product of sets. The in-neighborhood of $uv \in V(D_1 \Box D_2)$ can be written as 
	$N^-(uv) = (N^-(u) \times \{v\}) \cup (\{u\} \times N^-(v))$.

	\subsection{Forced Vertices}

    We begin by considering forced vertices in Cartesian products of orientations. We prove that they can be determined from the forced vertices of the orientations the Cartesian product is comprised of. We first focus on Cartesian products of any orientations.

    \begin{theorem}\label{thm-cart-forced}
    \sloppy
		Let $D_i$, $i=1,\dots,k$, be orientations. A vertex $u = u_1 \cdots u_k \in V(D_1 \Box \cdots \Box D_k)$ is a forced vertex if and only if there exists some $i \in \{1, \ldots , k\}$ such that $u_i$ is a forced vertex in $D_i$ and $u_j$ is a source for all $j \neq i$.
	\end{theorem}
	\begin{proof}
		($\Leftarrow$) If $\deg^-_{D_i} (u_i) \leq 1$, then $\deg^-(u_1 \cdots u_k) \leq 1$, and the claim follows from Corollary~\ref{cor:forceddegleq1}.
		
		Assume then that $\deg^-_{D_i} (u_i) \geq 2$ and without loss of generality that $i=1$. Since~$u_j$ is a source for all $j \geq 2$, we have $N^-(u_1 \cdots u_k) = N^-(u_1) \times \{u_2 \cdots u_k\}$. According to Theorem~\ref{thm:forced-characterization}, $u_1$ being forced implies that there exists a vertex $v_1 \neq u_1$ such that $N^-(u_1) \subseteq N^-[v_1]$. Now, $N^-(u_1 \cdots u_k) \subseteq N^-[v_1 u_2 \cdots u_k]$ and the claim follows due to Theorem~\ref{thm:forced-characterization}.
		
		($\Rightarrow$) Let $u = u_1 \cdots u_k$ be a forced vertex in $D_1 \Box \cdots \Box D_k$. According to Theorem~\ref{thm:forced-characterization}, $N^-(u) \subseteq N^-[v]$ for some $v \in V(D_1 \Box \cdots \Box D_k)$, $v \neq u$. Denote $v = v_1 \cdots v_k$.
		If $\deg^-(u) \leq 1$, then either all $u_i$ are sources or one $u_i$ has in-degree~1 and $u_j$ are sources for all $j \neq i$. In this case, $u_i$ is a forced vertex in $D_i$ and the claim follows. So assume that $\deg^-(u) \geq 2$. We have the following two possibilities:
		\begin{itemize}
			\item Assume that $\deg^-(u_i) \leq 1$ for all $i \in \{1, \ldots, k\}$. In order for $u$ to have in-degree at least~$2$, we must have $\deg^- (u_i) = 1$ for at least two~$u_i$, say, $\deg^- (u_1) = 1$ and $\deg^- (u_2) = 1$. Let $w_1$ and $w_2$ be in-neighbors of $u_1$ and $u_2$, respectively. The vertices $x_1=w_1 u_2 \cdots u_k$ and $x_2=u_1 w_2 u_3 \cdots u_k$ are in-neighbors of $u$ in the Cartesian product. Thus, they both must be in $N^- [v]$. Since $x_1$ and $x_2$ are neither in- nor out-neighbors of each other, $v \neq x_1$ and $v \neq x_2$. Thus, $x_1, x_2 \in N^-(v)$. 
			
			Since $x_1 \in N^-(v)$, we must have $v_1=w_1$ or $v_2=u_2$ (if neither holds, $x_1$ is neither an in- nor out-neighbor of $v$). Similarly, since $x_2 \in N^-(v)$, we must have $v_1=u_1$ or $v_2=w_2$. Therefore, either $v_1v_2=w_1w_2$ or $v_1v_2=u_1u_2$. Now, the component vertices of $v$ differ from those of $x_1$ and $x_2$ already in $v_1$ or $v_2$, and the remaining coordinates must all be equal, that is, $v_i=u_i$ for all $i \geq 3$. Recall that $v \neq u$, and thus $v = w_1 w_2 u_3 \cdots u_k$. However, now $x_1$ and $x_2$ are in fact out-neighbors of $v$, a contradiction.
			
			\item Assume that $\deg^-(u_i) \geq 2$ for some $i \in \{1, \ldots , k\}$, say, $\deg^-(u_1) \geq 2$. Denote the in-neighbors of $u_1$ by $t_\ell$ where $\ell \in \{1, \ldots , \deg^- (u_1) \}$.
			\begin{itemize}
				\item Suppose to the contrary that some $u_j$, $j \neq 1$ is not a source, say, $\deg^-(u_2) \geq 1$. Let $s \in N^-(u_2)$. As in the previous case, there are two in-neighbors $x_1 = t_1 u_2 \cdots u_k$ and $x_2 = u_1 s u_3 \cdots u_k$ of $u$ that lead to a contradiction. Thus, $\deg^-(u_j)=0$ for all $j \neq 1$.
				
				\item We will then show that $u_1$ is forced in $D_1$. Consider in-neighbors of $u$ that are of the form $x_\ell = t_\ell u_2 \cdots u_k$. In order for $x_\ell$ to be in $N^- [v]$ for all $\ell$, the component vertices of $v$ can differ only in one place when compared to each $x_\ell$. Thus, $v_1 = t_\ell$ for some $\ell$ or $v_j = u_j$ for all $j \geq 2$. If, say, $v_1 = t_1$, then $x_2 \in N^-[v]$ implies that $v_j = u_j$ for all $j \geq 2$. Thus, we must have $v_j = u_j$ for all $j \geq 2$ in all cases. However, now $t_\ell \in N^-[v_1]$ for all $\ell$, and thus $N^-(u_1) \subseteq N^-[v_1]$ which implies that $u_1$ is forced in $D_1$. \qedhere
			\end{itemize}
		\end{itemize}
	\end{proof}
	
	In Subsection~\ref{sec-cartpathcyc}, we focus on certain orientations of paths and cycles. The property they all have in common is that they do not have forced vertices of in-degree more than~1. The forced vertices in Cartesian products of such orientations can be easily determined from the in-degrees of vertices as described in the following corollary.

    \begin{corollary}\label{cor-cartforced}
		Let $D_i$, $i=1,\ldots,k$, be orientations such that $\deg^- (u) \leq 1$ for all forced vertices $u \in V(D_i)$.
		The vertex $u_1 \cdots u_k \in V(D_1 \Box \cdots \Box D_k)$ is a forced vertex in $D_1 \Box \cdots \Box D_k$ if and only if $\sum_{i=1}^{k} \deg^-(u_i) \leq 1$.
	\end{corollary}
	\begin{proof}
    \sloppy
		The vertex $u_1 \cdots u_k \in V(D_1 \Box \cdots \Box D_k)$ is a forced vertex if and only if $\deg^-(u_1 \cdots u_k) \leq 1$ by Theorem~\ref{thm-cart-forced} and Corollary~\ref{cor:forceddegleq1}. Since 
		\[N^-(u_1 \cdots u_k) = (N^-(u_1) \times \{u_2 \cdots u_k\}) \cup \cdots (\{u_1 \cdots u_{k-1}\} \times N^-(u_k)),\]
		we have $\deg^-(u_1 \cdots u_k) \leq 1$ if and only if $\sum_{i=1}^{k} \deg^-(u_i) \leq 1$.
	\end{proof}

    Recall that $n^-_i(D)$ denotes the number of vertices that have in-degree~$i$ in~$D$. With this notation, we can formulate the number of forced vertices in the case of Corollary~\ref{cor-cartforced} as follows.
    
	\begin{corollary}\label{cor-cartforcednumber}
		Let $D_i$, $i=1,\ldots,k$, be orientations such that $\deg^- (u) \leq 1$ for all forced vertices $u \in V(D_i)$. The number of forced vertices in $D_1 \Box \cdots \Box D_k$ is
		\[ \sum_{\mathbf{x}\in X} \prod_{i=1}^{k} n_{x_i}^- (D_i), \]
		where $X = \{(x_1,\ldots,x_k) \in \{0,1\}^k \mid \sum_{i=1}^{k} x_i \leq 1\}$.
	\end{corollary}
    
	\subsection{Cartesian Products of Oriented Paths and Cycles}\label{sec-cartpathcyc}
	
	In this section, we consider arbitrary Cartesian products of oriented paths and cycles. The paths we consider have at least two vertices and the cycles have at least five vertices. Additionally, the cycles are oriented such that they have at least two sources each. We will prove that the solid-metric dimension of the Cartesian product $D_1 \Box \cdots \Box D_k$ of such paths and cycles is equal to the number of forced vertices (see Corollary~\ref{cor-cartforcednumber}). The assumptions for the number of vertices in the orientations we consider enable us to use Corollaries~\ref{cor-cartforced} and~\ref{cor-cartforcednumber}.
	
	The proof is largely based on which vertices are reached from which forced vertices. We denote the set of vertices that are reachable from $v$ by $R(v)$. In a Cartesian product $D_1 \Box \cdots \Box D_k$, we can consider the reachability between vertices through the component graphs of the product. Denote $u = u_1 \cdots u_k$ and $v = v_1 \cdots v_k$, where $u_i,v_i \in V(D_i)$. There exists a directed path from $v$ to $u$ if and only if each $u_i$ is reachable from each $v_i$.
    Thus, $R(v) = R(v_1) \times \cdots \times R(v_k)$. 
	Moreover, much like in undirected Cartesian products we have
	\[ d(v,u) = \sum_{i=1}^{k} d(v_i,u_i), \]
	where if $d(v_i,u_i) = \infty$ for some $i$, then we agree that $d(v,u) = \infty$.

	\begin{theorem}\label{thm-cart-main}
		Let $D_i$ be an oriented path with at least two vertices or an oriented cycle with at least five vertices and two sources for $i \in \{1, \ldots , k\}$. We have
		\[ \beta_S(D_1 \Box \cdots \Box D_k) =  \sum_{\mathbf{x}\in X} \prod_{i=1}^{k} n_{x_i}^- (D_i), \]
        where $X = \{(x_1,\ldots,x_k) \in \{0,1\}^k \mid \sum_{i=1}^{k} x_i \leq 1\}$
	\end{theorem}
	\begin{proof}
		Due to Corollary~\ref{cor-cartforcednumber}, we only need to show that the forced vertices do indeed form a solid-resolving set of the Cartesian product. Denote the set of forced vertices by $S$. By Theorem~\ref{thm:srsmain}, it suffices to prove that for all distinct vertices $u$ and $v$, there exists an $s \in S$ such that $d(s,u) < d(s,v)$.
		
		Let $u=u_1 \cdots u_k$ be an arbitrary vertex. If $u \in S$, then clearly for all $v \neq u$, we have $d(u,u) < d(u,v)$. Assume that $u \notin S$. Due to the assumptions on $D_i$, every vertex in~$D_i$ is reachable from some source in $D_i$. This implies that $u$ is reachable from a source $x = x_1 \cdots x_k$, where $u_i$ is reachable from $x_i$ in each $D_i$. We assume to the contrary that there exists a vertex $v \neq u$ such that $d(s,v) \leq d(s,u)$ for all $s \in S$. In particular, $v$ is reachable from every vertex in $S$ that reaches $u$.
		
		Since the component graphs $D_i$ do not have forced vertices with in-degree~2 or more, Corollary~\ref{cor-cartforced} implies that either at least one of the vertices $u_i$ is a sink of in-degree~2 or at least two of the vertices $u_i$ are of in-degree~1. We will prove that for such indices $i$ either $v_i=u_i$ or there exists a source $x_i$ in $D_i$ such that $d(x_i,u_i) < d(x_i,v_i)$. 
		
		Assume $u_i$ is of in-degree~1. Now $x^i=x_1 \cdots x_{i-1} u_i x_{i+1} \cdots x_k$ is a forced vertex and $u$ is reachable from $x^i$. Moreover, for all $t_i$ in $D_i$, $u$ is reachable from $x_1 \cdots x_{i-1} t_i x_{i+1} \cdots x_k$ if and only if $t_i$ lies on the directed path from $x_i$ to $u_i$ (and then these are also forced vertices). Since $d(s,v) \leq d(s,u)$ for all $s \in S$, $v$ must also be reachable from the aforementioned forced vertices. This implies that either $v_i = u_i$ or a directed path from $x_i$ to $v_i$ goes through $u_i$ (in $D_i$), and thus $d(x_i,u_i) < d(x_i,v_i)$. 
		
		Assume then that, say, $u_1$ is a sink of in-degree~2. Now $u$ is reachable from two sources: $x_\ell = x_1^\ell x_2 \cdots x_k$ and $x_r = x_1^r x_2 \cdots x_k$, where $x_1^\ell$ and $x_1^r$ are the two sources in $D_1$ that reach~$u_1$. Now $v_1$ must also be reached from both of these sources, which implies that $v_1$ is a sink of in-degree~2. Assume that $v_1 \neq u_1$. Then $D_1$ is a cycle with exactly two sources and $u_1$ and $v_1$ are the two sinks. However, since the cycle has at least five vertices, there exists a vertex $t_1$ of in-degree~1. Now, if $u_1$ is reachable from $t_1$, then $v_1$ is not. This implies that there is a forced vertex in the Cartesian product that reaches $u$ but not $v$, a contradiction. Thus, no in-degree~1 vertex in $D_1$ reaches $u_1$. Thus, the two sources are both in-neighbors of $u_1$ and we have $d(x_1^\ell, u_1) < d(x_1^\ell, v_1)$ or $d(x_1^r, u_1) < d(x_1^r, v_1)$. Therefore, either $v_1=u_1$ or there exists a source $x_1$ in $D_1$ such that $d(x_1,u_1) < d(x_1,v_1)$.

		Without loss of generality, assume that $u_i$ is a sink of in-degree~2 for $1 \leq i \leq a$, has in-degree~1 for $a < i \leq b$, and is a source $x_i$ for $b < i \leq k$, where $0 \leq a \leq b \leq k$. By the considerations above, when $1 \leq i \leq b$, either $v_i=u_i$ or there exists a source $x_i$ such that $d(x_i,u_i) < d(x_i,v_i)$. Let $x_i$ be such sources if they exist. Now,
		\begin{align*} d(x,u) & = \sum_{i=1}^{k} d(x_i,u_i) = \sum_{i=1}^{b} d(x_i,u_i) + \sum_{i=b+1}^{k} d(x_i,x_i)  \leq \sum_{i=1}^{b} d(x_i,v_i)  \leq d(x,v).
		\end{align*}
		Equality in the above implies that $v_i = u_i$ for all $i \in \{1, \ldots, k\}$, and thus $v=u$, a contradiction. Therefore, we must have $d(x,u) < d(x,v)$, a contradiction to our assumption.
	\end{proof}

\begin{remark}
    The restriction on the oriented cycle in Theorem~\ref{thm-cart-main} is necessary. For example, if $D$ is an orientation of $C_n$ such that it has no sources (i.e. $v_1 \rightarrow v_2 \rightarrow \cdots \rightarrow v_n \rightarrow v_1$), then $D \Box D$ has no forced vertices and we naturally cannot use them to determine the solid-metric dimension. If $D$ is an orientation of $C_n$ such that it has exactly one source, the Cartesian product $D \Box D$ may contain forced vertices but they might not form a solid-resolving set of $D \Box D$. For example, if $D$ is an orientation of $C_5$ such that $v_1 \rightarrow v_2 \rightarrow v_3 \rightarrow v_4 \rightarrow v_5 \leftarrow v_1$, then $D \Box D$ has seven forced vertices but $\beta_S (D \Box D) = 8$.
\end{remark}

\section{Conclusions}
\label{sec-conclusions}
This work acts as an introduction to the study of solid-resolving sets of digraphs. We have defined this new concept along with many related concepts such as solid-metric dimension, forced vertices, and boundary vertices. We used these basic definitions to obtain several results for complete graphs, trees, cycles, and complete bipartite graphs. We also showed through various examples that the solid-metric dimension of a digraph does not have any trivial relation with the number of arcs, vertices, or the solid-metric dimension of its underlying undirected graph. We also characterized the solid-metric dimension of the Cartesian product of certain digraphs based on the solid-metric dimension of their components.

As with any introductory work, this paper opens up many possible directions of research. To begin with, recall that we were not able to determine the lower bounds for the solid-metric dimensions of orientations of complete graphs and complete bipartite graphs. We also considered only orientations of cycles, rather than all digraphs whose underlying graphs are cycles. Finding a closed formula for these, as well as for other simple graph families, could be interesting.
It would also be interesting to prove or disprove the existence of any relation between the solid-metric dimension of a digraph as well as its or its underlying graph's metric dimension, or any other related parameter. Additionally, the effect of arc or vertex deletion on the metric dimension of a digraph seems to be unexplored. In particular, it would be interesting to see whether 1 is the largest possible increase when one arc is removed.

Since solid-resolving sets for a digraph are a generalization of the solid-resolving sets of undirected graphs, we are interested in some problems from the perspective of real-life applications. For instance, can we characterize the graphs for which the solid-metric dimension is either the minimum or the maximum possible? Moreover, can we say something about the algorithmic problem of determining the solid-metric dimension in terms of parameterized complexity? Determining the graph classes for which the problem can be solved in polynomial time is also of interest.

\end{document}